\documentclass[leqno,11pt]{amsart}
\usepackage{amsmath,amstext,amssymb,amsopn,amsthm,mathrsfs,wasysym,dsfont,comment,xcolor,cite,enumitem,tikz}

\theoremstyle{plain}

\usetikzlibrary{arrows}
\usetikzlibrary{patterns}

\pgfdeclarepatternformonly{my dots}{\pgfqpoint{-1pt}{-1pt}}{\pgfqpoint{5pt}{5pt}}{\pgfqpoint{6pt}{6pt}}
{
    \pgfpathcircle{\pgfqpoint{0pt}{0pt}}{0pt}
    \pgfpathcircle{\pgfqpoint{5pt}{3pt}}{.4pt}
    \pgfusepath{fill}
}
\allowdisplaybreaks

\newtheorem{thm}{Theorem}[section]
\newtheorem{cor}[thm]{Corollary}
\newtheorem{pro}[thm]{Proposition}
\newtheorem{ob}[thm]{Observation}

\newtheorem{lem}[thm]{Lemma}
\newtheorem{rem}[thm]{Remark}
\newtheorem*{ex*}{Example}

\def\1{\boldsymbol{1}}

\def\8{\infty}

\def\8{\infty}
\DeclareMathOperator{\support}{supp}
\DeclareMathOperator{\loc}{loc}

\DeclareMathOperator{\pv}{p.v.}

\title[Harmonic analysis related to exotic Bessel operators]
	{Mapping properties of fundamental harmonic analysis operators in the exotic Bessel framework}

\author[B. Langowski]{Bartosz Langowski}
\author[A. Nowak]{Adam Nowak}

\address{Bartosz Langowski,\newline
Indiana University, Department of Mathematics, \newline
831 East 3rd St., Bloomington, IN 47405, USA
\newline
\indent and \newline              
Wroc\l{}aw University of Science and Technology,
Faculty of Pure and Applied Mathematics,
\newline
Wyb.\ Wyspia\'nskiego 27, 50--370 Wroc\l{}aw, Poland }
\email{balango@iu.edu}

\address{Adam Nowak, \newline			
		Polish Academy of Sciences, 
		Institute of Mathematics, \newline
      \'Sniadeckich 8,
      00--656 Warszawa, Poland    
      }
\email{anowak@impan.pl}

\begin{document}
%\date{\today}

\begin{abstract}
We prove sharp power-weighted $L^p$, weak type and restricted weak type inequalities for the heat semigroup maximal
operator and Riesz transforms associated with the Bessel operator $B_{\nu}$ in the exotic range of the parameter
$-\infty < \nu < 1$. Moreover, in the same framework, we characterize basic mapping properties for other
fundamental harmonic analysis operators, including the heat semigroup based vertical $g$-function and
fractional integrals (Riesz potential operators).
\end{abstract}

\maketitle
\thispagestyle{empty}

\footnotetext{
%42C05 - Orthogonal functions and polynomials, general theory of nontrigonometric harmonic analysis 
%42C20 - Other transformations of harmonic type
%42C99 - None of the above, but in this section

Research supported by the National Science Centre of Poland within the project OPUS 2017/27/B/ST1/01623.
}

%%%%%%%%%%%%%%%%%%%%%%%%%%%%%%%%%%%%%%%%%%%%%%%%%%%%%%%%%%%%%%
%%%%%%%%%%%%%%%%%%%%%%%%%%%%%%%%%%%%%%%%%%%%%%%%%%%%%%%%%%%%%%

\section{Introduction} \label{sec:intro}

Let
$$
B_{\nu} = - \frac{d^2}{dx^2} - \frac{2\nu+1}{x}\frac{d}{dx} 
$$
be the Bessel differential operator. For $\nu \in \mathbb{R}$ we consider $B_{\nu}$ acting on functions
on $\mathbb{R}_+ = (0,\infty)$. This operator is formally symmetric in $L^2(d\eta_{\nu})$, where
$$
d\eta_{\nu}(x) = x^{2\nu+1}\, dx, \qquad x > 0.
$$

When $\nu > -1$, there exists a classical self-adjoint extension $B_{\nu}^{\textrm{cls}}$ of $B_{\nu}$
(acting initially on $C_c^2(\mathbb{R}_+)$), whose spectral decomposition is given via the (modified) Hankel transform
(see e.g.\ \cite[Section 4]{NSS}). Harmonic analysis related to $B_{\nu}^{\textrm{cls}}$,
having roots in the seminal work of Muckenhoupt and Stein \cite{MuSt}, has been very widely studied and now is
well understood. The related bibliography seems to be endless. For some recent developments, see for instance
\cite{Bet12,Bet14,Bet13,Bet11,Bet15,Bet16,BHNV,CaSz,DPW,KaPr,NoSt1}
and references therein. But this by no means exhausts contributions to the subject.
Note that the metric measure space $(\mathbb{R}_+,|\cdot|,d\eta_{\nu})$, where $|\cdot|$ stands for the Euclidean norm, is
a space of homogeneous type when $\nu > -1$, in particular the measure $d\eta_{\nu}$ is doubling.
Note also that for $\nu=n/2-1$ the Bessel operator is just the radial part of $-\Delta$ in $\mathbb{R}^n$, $n \ge 1$,
so for these $\nu$ analysis related to $B_{\nu}$ corresponds to radial analysis in $\mathbb{R}^n$.

On the other hand, the case $\nu \le -1$ is much less recognized. It turns out that for these $\nu$, or even for the slightly
larger range $\nu < 1$, there exists a self-adjoint extension of $B_{\nu}$ (considered initially on $C_c^2(\mathbb{R}_+)$)
expressible in terms of the (modified) Hankel transform, but in a different way from $B_{\nu}^{\textrm{cls}}$ if $\nu \neq 0$.
This was established only recently \cite{NSS}. We call this new self-adjoint operator $B_{\nu}^{\textrm{exo}}$
(from exotic), and refer to the corresponding transform as exotic (modified) Hankel transform. See \cite[Section 4]{NSS}
for details. Note that for $\nu \le -1$ the measure $\eta_{\nu}$ is not even locally finite: there are balls
in $(\mathbb{R}_+,|\cdot|,d\eta_{\nu})$ of arbitrarily
small radii and infinite measure. Hence the situation does not fall under classical theories/setups, like e.g.\
Calder\'on-Zygmund operator theory on spaces of homogeneous/non-homogeneous type.
According to our best knowledge, harmonic analysis aspects of the exotic Bessel context were studied so far only
in \cite{NSS} and in the very recent paper \cite{Ka}.

The aim of this article is to study in detail mapping properties of several fundamental harmonic analysis operators
in the exotic Bessel framework. For the exotic Bessel semigroup maximal operator and Riesz transforms associated with
the exotic Bessel operator we obtain a complete characterization of power-weighted $L^p$, weak type and restricted
weak type inequalities, see Theorem \ref{thm:maxWexo} and Theorems \ref{thm:Rexo} and \ref{thm:Rexob}, respectively.
A similar result is concluded for the Bessel-Poisson semigroup maximal operator as well, see Proposition \ref{prop:Poisson}.
Another operator we consider is the vertical $g$-function based on the exotic Bessel semigroup. In this case
we characterize power-weighted $L^p$ and restricted weak type inequalities, and get almost sharp description of
power-weighted weak type inequalities, see Theorem \ref{thm:gexo}. We also treat fractional integrals, viz.\ Riesz potential
operators, related to the exotic Bessel operator. For these operators, we derive a complete characterization of
two-weight $L^p-L^q$ inequalities with power weights involved, see Theorem \ref{thm:potstr}.
Let us emphasize that sharpness is an important aspect of all these results. Likewise, it is an important ingredient
of our motivation, since sharp description of basic mapping properties enables deep insight into the nature of
the operators in question.

Inspiration for this research comes mainly from the paper \cite{BHNV}, and also from \cite{NoSt1}, where
analogous results were obtained in the classical Bessel setting. In fact, here we make use of the results, analysis and
methods applied there. But we also go further, with some new analysis. The latter pertains the most to the exotic
Riesz-Bessel transforms. For these operators even the very definition is not straightforward. Moreover,
there is no handy relation between the classical and exotic Riesz transforms, in contrast with the other operators
considered. Roughly speaking, this is due to the spatial variable differentiation entering the definitions of
the Riesz transforms. For the same reason, there is no coincidence (or rather a discontinuity occurs)
with the classical Riesz-Bessel transforms as the exotic parameter $\nu$ tends to $0$ (note that
$B_{\nu}^{\textrm{cls}}=B_{\nu}^{\textrm{exo}}$ for $\nu=0$). This phenomenon is, perhaps, a bit unexpected
and makes a difference from the other operators investigated.

In this work we consider only the one-dimensional situation and put emphasis on sharpness and completeness of the results.
Nevertheless, there exists a wider background of our setup. First of all, one can construct a general multi-dimensional
(exotic) Bessel framework in a natural way, simply by taking tensor products of
the one-dimensional situations related to $B_{\nu}^{\textrm{cls}}$ and $B_{\nu}^{\textrm{exo}}$, cf.\ \cite{NSS}.
In this way, some coordinates (axes) may be classical and some exotic. Also, there is a Dunkl theory generalization of
the {exotic} Bessel concept, see \cite{ALN}. Similar exotic situations naturally emerge in settings associated with classical
orthogonal expansions, see \cite{NSS}. In each case, the underlying philosophy and basic demand are the same:
to admit all values of the associated parameters in the initial pair
[differential or difference-differential `Laplacian', the associated measure] and subsequent analysis.

It is worth mentioning that there is a probabilistic interpretation of the exotic Bessel setting.
The semigroup generated by $B_{\nu}^{\textrm{cls}}$, $\nu > -1$, is the transition semigroup of the Bessel diffusion process
$X^{\nu}_{2t}$, which is well known in the literature. The semigroup generated by $B_{\nu}^{\textrm{exo}}$, $\nu < 0$, is
also a transition semigroup of another Bessel process $\widetilde{X}^{\nu}_{2t}$, which in the overlapping range $-1 < \nu < 0$ is
different from $X^{\nu}_{2t}$. The two processes are dual to each other in a certain sense, moreover $\widetilde{X}^{\nu}$ is Doob's
$h$-transform of $X^{\nu}$ with $h(x)= x^{-2\nu}$.
For all this and further facts on Bessel processes see e.g.\ \cite[p.\ 33--35, 71--76, 133--134]{BS} and \cite[Chapter XI]{RY}.

%%%%%%%%%%%%%%%%%%%%%%%%%%%%%%%%%%%%%%%%%%%%%%%%%%%%%%%%%%%%%%

\subsection*{Structure of the paper}
Below, still in this section, we comment on the notation used in the paper and recall basic terminology needed.
Section \ref{sec:tech} constitutes a technical preparation needed later on. In particular, it invokes a number
of auxiliary `special' operators and summarizes their fundamental mapping properties. In Section \ref{sec:max}
the exotic Bessel semigroup maximal operator is studied, whereas in Section \ref{sec:riesz} Riesz transforms
in the exotic Bessel framework are defined and investigated. Sections \ref{sec:g} and \ref{sec:frac} are devoted
to the vertical $g$-function based on the exotic Bessel semigroup and fractional integrals in the exotic Bessel
situation, respectively.

%%%%%%%%%%%%%%%%%%%%%%%%%%%%%%%%%%%%%%%%%%%%%%%%%%%%%%%%%%%%%%

\subsection*{Notation}
Throughout the paper we use a fairly standard notation. Thus $\mathbb{R}_+ = (0,\infty)$.
For the sake of brevity, we often omit $\mathbb{R}_{+}$ when denoting $L^p$, or more generally Lorentz spaces $L^{p,q}$,
related to the measure spaces $(\mathbb{R}_{+}, x^{\delta}dx)$ and $(\mathbb{R}_+, d\eta_{\nu})$. For instance,
$L^p(x^{\delta}dx) = L^p(\mathbb{R}_{+}, x^{\delta}dx)$.
As usual, for $1 \le p \le \infty$, $p'$ denotes its conjugate exponent, $1/p+1/p' = 1$.
By weakening a strict inequality ``$<$'' we mean replacing it by ``$\le$''. Similarly, strictening a weak 
inequality ``$\le$'' means replacing it by ``$<$''.
We write $X \lesssim Y$ to indicate that $X \le CY$ with a positive constant $C$ independent
of significant quantities. We shall write $X \simeq Y$ when simultaneously $X \lesssim Y$ and $Y \lesssim X$.

%%%%%%%%%%%%%%%%%%%%%%%%%%%%%%%%%%%%%%%%%%%%%%%%%%%%%%%%%%%%%%

\subsection*{Basic terminology}
The notions, facts and terminology that follow have a more general meaning, but here they
are specified to $\mathbb{R}_+$ equipped with a power measure, and power weights, since this is what we need in this paper.

Let $1 \le p < \infty$ and $\delta \in \mathbb{R}$.
An operator $T$ is said to be of strong type $(p,p)$ with respect to the measure space $(\mathbb{R}_+,x^{\delta}dx)$
when it is bounded on $L^p(\mathbb{R}_+,x^{\delta}dx)$. Strong type $(\infty,\infty)$ of $T$ means
boundedness of $T$ on $L^{\infty}(\mathbb{R}_+,x^{\delta}dx) = L^{\infty}(\mathbb{R_+},dx)$.
Further, $T$ is said to be of weak type $(p,p)$ with respect to the measure space $(\mathbb{R}_+,x^{\delta}dx)$
if it satisfies the weak type $(p,p)$ estimate
$$
\lambda^p \int_{\{y > 0 : |Tf(y)|> \lambda\}} x^{\delta}\, dx \lesssim \int_0^{\infty} |f(x)|^p x^{\delta}\, dx,
	\qquad \lambda > 0,
$$
uniformly in $f \in L^p(\mathbb{R}_+,x^{\delta}dx)$. This is equivalent to boundedness
of $T$ from $L^p(\mathbb{R}_+,x^{\delta}dx)$ to the Lorentz space $L^{p,\infty}(\mathbb{R}_+,x^{\delta}dx)$.
The latter space is referred to as weak $L^p(\mathbb{R}_+,x^{\delta}dx)$.
Finally, $T$ is of restricted weak type $(p,p)$ with respect to the measure space $(\mathbb{R}_+,x^{\delta}dx)$
if it satisfies the weak type $(p,p)$ estimate after restricting to $f$ being characteristic functions of sets of finite
$x^{\delta}dx$ measure. This property is equivalent to boundedness of $T$ between the extreme Lorentz spaces,
from $L^{p,1}(\mathbb{R}_+,x^{\delta}dx)$ to $L^{p,\infty}(\mathbb{R}_+,x^{\delta}dx)$. Recall that, on the second
index scale, the space $L^{p,1}$ is the smallest one, and $L^{p,\infty}$ is the biggest one among $L^{p,q}$, $1\le q \le \infty$.

Let now $1 \le p \le \infty$ and $A \in \mathbb{R}$. As before,
we denote by $L^p(\mathbb{R}_+,x^{Ap}d\eta_{\nu})$ the $x^{Ap}$ power
weighted $L^p$ space with respect to the $d\eta_{\nu}$ measure. This means that $f \in L^p(\mathbb{R}_+,x^{Ap}d\eta_{\nu})$
if and only if $x^{A}f \in L^p(\mathbb{R}_+,d\eta_{\nu})$. The point is that this way of writing weights allows one to introduce
them also in the $L^{\infty}$ case without violating much the general notation. Thus, by convention,
$L^{\infty}(\mathbb{R}_+,x^{A\infty}d\eta_{\nu})$ consists of all measurable
functions $f$ such that $x^A f$ is essentially bounded on $\mathbb{R}_+$, and the norm of $f$ in that
space is $\|x^A f\|_{\infty}$.

%%%%%%%%%%%%%%%%%%%%%%%%%%%%%%%%%%%%%%%%%%%%%%%%%%%%%%%%%%%%%%
%%%%%%%%%%%%%%%%%%%%%%%%%%%%%%%%%%%%%%%%%%%%%%%%%%%%%%%%%%%%%%

\section{Technical preparation} \label{sec:tech}

In this section we gather facts, formulas, and results that will be needed to prove our main results.
Most of this material either comes from the existing literature or can easily be deduced from there.

%%%%%%%%%%%%%%%%%%%%%%%%%%%%%%%%%%%%%%%%%%%%%%%%%%%%%%%%%%%%%%
\subsection{Bessel functions} \label{ssec:Bes}
Facts and formulas presented in this subsection can be found, e.g., in \cite{Wat,Leb,handbook,lib}.
An important object in our study is the modified Bessel function $I_{\mu}$ of order $\mu$, which in this paper
is always considered as a function on $\mathbb{R}_+$, and the order is always assumed to satisfy $\mu > -1$.

The function $I_{\mu}(w)$ is strictly positive and smooth for $w>0$.
It has the series expansion
\begin{equation} \label{bes:ser}
I_{\mu}(w) = \sum_{n=0}^{\infty} \frac{(w/2)^{2n+\mu}}{\Gamma(n+1)\Gamma(n+\mu+1)},
\end{equation}
from which it is easily seen that
\begin{equation} \label{bes:lim}
w^{-\mu}I_{\mu}(w)\big|_{w=0^+} := \lim_{w \to 0^+} w^{-\mu}I_{\mu}(w) = \frac{1}{2^{\mu}\Gamma(\mu+1)}.
\end{equation}
Also, it follows that for small $w$
\begin{equation} \label{bes:as0}
I_{\mu}(w) = \frac{1}{2^{\mu}\Gamma(\mu+1)} w^{\mu} + \mathcal{O}\big( w^{\mu+2}\big).
\end{equation}
On the other hand, the large argument asymptotic is
\begin{equation} \label{bes:inf}
I_{\mu}(w) = {e^{w}}\bigg(\frac{1}{\sqrt{2\pi w}} + \mathcal{O}\Big( \frac{1}{w^{3/2}}\Big) \bigg).
\end{equation}
Thus, in particular, for any fixed $A>0$
\begin{equation} \label{bes:est}
I_{\mu}(w) \simeq w^{\mu}, \quad w \in (0,A] \qquad \textrm{and} \qquad
I_{\mu}(w) \simeq w^{-1/2}e^w, \quad w \in [A,\infty).
\end{equation}

The differentiation rule for $I_{\mu}$ is
\begin{equation} \label{bes:dif}
\frac{d}{dw} \big( w^{-\mu} I_{\mu}(w) \big) = z^{-\mu} I_{\mu+1}(w).
\end{equation}
Another fundamental formula is the order recurrence relation
\begin{equation} \label{bes:rec}
\frac{2\mu}{w} I_{\mu}(w) = I_{\mu-1}(w) - I_{\mu+1}(w), \qquad \mu > 0.
\end{equation}

Finally, note that only for odd half-integer orders $I_{\mu}$ can be expressed directly via elementary functions.
In particular,
\begin{equation} \label{bes:el}
I_{-1/2}(w) = \sqrt{\frac{2}{\pi w}} \cosh w, \qquad I_{1/2}(w) = \sqrt{\frac{2}{\pi w}} \sinh w.
\end{equation}

%%%%%%%%%%%%%%%%%%%%%%%%%%%%%%%%%%%%%%%%%%%%%%%%%%%%%%%%%%%%%%
\subsection{Hardy type operators} \label{ssec:Hardy}

For a parameter $\xi \in \mathbb{R}$, consider the following Hardy type operator and its dual:
\begin{align*}
H_0^{\xi}f(x) & = x^{-\xi-1} \int_0^{x} f(y) y^{\xi}\, dy, \qquad x > 0, \\
H_{\infty}^{\xi}f(x) & = x^{\xi} \int_x^{\infty} f(y) y^{-\xi-1}\, dy, \qquad x > 0.
\end{align*}
Mapping properties of $H_0^{\xi}$ and $H_{\infty}^{\xi}$ are essential for our developments.
The next two lemmas give characterizations of power-weighted strong, weak and restricted weak type boundedness of
$H_0^{\xi}$ and $H_{\infty}^{\xi}$.
Weaker statements, providing only sufficiency parts in restricted ranges of $\xi$, can be found, e.g., in
\cite{CRH,HSTV,BHNV}. Nonetheless, crucial parts of Lemmas \ref{lem:H0} and \ref{lem:Hinf} are rather straightforward
consequences of results found in \cite{AM}, see also references given there.

\begin{lem} \label{lem:H0}
Let $\xi,\delta \in \mathbb{R}$ and $1 \le p < \infty$.
Consider $H_0^{\xi}$ on the measure space $(\mathbb{R}_+,x^{\delta}dx)$. Then
\begin{itemize}
\item[(a)] $H_0^{\xi}$ is of strong type $(p,p)$ if and only if $\delta < (\xi+1)p-1$;
\item[(b)] $H_0^{\xi}$ is of weak type $(p,p)$ if and only if $\delta < (\xi+1)p-1$, with the inequality weakened in case
	$p=1$ and $\xi \neq -1$;
\item[(c)] $H_0^{\xi}$ is of restricted weak type $(p,p)$ if and only if 
	$\delta \le (\xi+1)p-1$, with the inequality strictened in case $\xi = -1$.
\end{itemize}
Moreover, $H_0^{\xi}$ is of strong type $(\infty,\infty)$ if and only if $\xi > -1$.
\end{lem}

\begin{lem} \label{lem:Hinf}
Let $\xi,\delta \in \mathbb{R}$ and $1 \le p < \infty$.
Consider $H_{\infty}^{\xi}$ on the measure space $(\mathbb{R}_+,x^{\delta}dx)$. Then
\begin{itemize}
\item[(a)] $H_{\infty}^{\xi}$ is of strong type $(p,p)$ if and only if $-\xi p -1 < \delta$;
\item[(b)] $H_{\infty}^{\xi}$ is of weak type $(p,p)$ if and only if $-\xi p -1 < \delta$, with the inequality
	weakened in case $p=1$ and $\xi \neq 0$;
\item[(c)] $H_{\infty}^{\xi}$ is of restricted weak type $(p,p)$ if and only if $-\xi p -1 \le \delta$, with
	the inequality strictened in case $\xi = 0$.
\end{itemize}
Moreover, $H_{\infty}^{\xi}$ is of strong type $(\infty,\infty)$ if and only if $\xi > 0$.
\end{lem}

For the sake of brevity, in the proofs of Lemmas \ref{lem:H0} and \ref{lem:Hinf} we shall tacitly use some
notation and terminology from \cite{AM}, like e.g.\ operators $P_{\xi}$ and $Q_{\xi}$.

\begin{proof}[{Proof of Lemma \ref{lem:H0}}]
Item (a) follows from \cite[Theorem A]{AM}. Indeed, it suffices to notice that
$H_{0}^{\xi}$ is of strong type $(p,p)$ with respect to $(\mathbb{R}_+,x^{\delta}dx)$ if and only if
$(x^{-(\xi+1)p+\delta},x^{-\xi p + \delta})$ is a strong type $(p,p)$ weight pair for $P_0$.

Item (b) is a consequence of \cite[Theorems 1,2]{AM}. To see this, observe that
$H_{0}^{\xi}$ is of weak type $(p,p)$ with respect to $(\mathbb{R}_+,x^{\delta}dx)$ if and only if
$(x^{\delta},x^{-\xi p + \delta})$ is a weak type $(p,p)$ weight pair for $P_{\xi+1}$.

Concerning item (c), necessity of the condition is shown by a simple counterexample. Let $f = \chi_{(1,2)}$.
Clearly, $f \in L^p(x^{\delta}dx)$ for any $\delta \in \mathbb{R}$. On the other hand,
$H_0^{\xi}f(x) \simeq x^{-\xi-1}$ for large $x$, which implies that $H_0^{\xi}f$ does not belong to weak $L^p(x^{\delta}dx)$
unless $\delta \le (\xi+1)p-1$ ($<$ in case $\xi=-1$).
It remains to check that $H_0^{\xi}$ is of restricted weak type $(p,p)$ with respect to $(\mathbb{R}_+,x^{\delta}dx)$
for $\delta = (\xi+1)p -1$, $\xi \neq -1$, $p > 1$. This can be done by means of H\"older's inequality
in Lorentz spaces. Indeed, we have
$$
\big|H_0^\xi f(x)\big| \le x^{-\xi-1}\int_0^\infty |f(y)|y^\xi \,dy \le x^{-\xi-1}\big\|x^{\xi-\delta}
	\big\|_{L^{p', \infty}(x^\delta dx)}	\|f\|_{L^{p, 1}(x^\delta dx)}
$$
and, consequently,
$$
\big\|H_0^\xi f\big\|_{L^{p, \infty}(x^\delta dx)}\le \big\|x^{-\xi-1}\big\|_{L^{p, \infty}(x^\delta dx)}
		\big\|x^{\xi-\delta} \big\|_{L^{p', \infty}(x^\delta dx)}\|f\|_{L^{p, 1}(x^\delta dx)}.
$$
Since $x^{-\xi-1}\in L^{p, \infty}(x^\delta dx)$ and $x^{\xi-\delta}\in L^{p', \infty}(x^\delta dx)$, we get the desired
boundedness.

Finally, the assertion about strong type $(\infty,\infty)$ is easily verified directly.
\end{proof}

\begin{proof}[{Proof of Lemma \ref{lem:Hinf}}]
The reasoning is parallel to that from the proof of Lemma \ref{lem:H0}.
Item (a) is deduced from \cite[Theorem B]{AM}, since $H_{\infty}^{\xi}$ is of strong type $(p,p)$ with respect to
$(\mathbb{R}_+,x^{\delta}dx)$ if and only if $(x^{\xi p + \delta}, x^{(\xi+1)p+\delta})$ is a strong type $(p,p)$
weight pair for $Q_0$. Likewise, item (b) follows from \cite[Theorems 4,5]{AM}, taking into account that
$H_{\infty}^{\xi}$ is of weak type $(p,p)$ with respect to $(\mathbb{R}_+,x^{\delta}dx)$ if and only if
$(x^{\delta},x^{(\xi+1)p+\delta})$ is a weak type $(p,p)$ weight pair for $Q_{-\xi}$.

Necessity in (c) follows by the same counterexample as before, $f = \chi_{(1,2)} \in L^p(x^{\delta}dx)$. Then
$H_{\infty}^{\xi}f(x) \simeq x^{\xi}$ for small $x > 0$, consequently $H_{\infty}^{\xi}f$ is not in weak $L^p(x^{\delta}dx)$
unless $\delta \ge -\xi p -1$ ($>$ in case $\xi=0$). Sufficiency in (c) reduces to checking that $H_{\infty}^{\xi}$
is of restricted weak type $(p,p)$ with respect to $(\mathbb{R}_+,x^{\delta}dx)$ for $\delta=-\xi p-1$, $\xi \neq 0$, $p>1$,
and this is obtained by H\"older's inequality in Lorentz spaces, similarly as in the proof of Lemma \ref{lem:H0}.

The assertion about strong type $(\infty,\infty)$ is again easily verified directly.
\end{proof}

We will also need variants of $H_0^1$ and $H_{\infty}^{-1}$ involving logarithms. Define
\begin{align*}
H_0^{1,\log}f(x) & = \frac{1}{x^2} \int_0^x \log\frac{x}y\, f(y)\, y\, dy, \qquad x > 0,\\
H_{\infty}^{-1,\log}f(x) & = \frac{1}{x} \int_x^{\infty} \log\frac{y}x\, f(y)\, dy, \qquad x > 0.
\end{align*}

\begin{lem} \label{lem:Hlog}
Let $\delta \in \mathbb{R}$ and $1 \le p < \infty$.
\begin{itemize}
\item[(a)]
For $\delta < 2p-1$,
$H_0^{1,\log}$ is bounded on $L^p(\mathbb{R}_+,x^{\delta}dx)$.
For $\delta \ge 2p-1$ it is not of restricted weak type $(p,p)$ with respect to $(\mathbb{R}_+,x^{\delta}dx)$.
\item[(b)]
For $\delta > p-1$, $H_{\infty}^{-1,\log}$ is bounded on $L^p(\mathbb{R}_+,x^{\delta}dx)$.
For $\delta \le p-1$ it is not of restricted weak type $(p,p)$ with respect to $(\mathbb{R}_+,x^{\delta}dx)$.
\end{itemize}
\end{lem}

\begin{proof}
Boundedness on $L^p$ as stated in (a) and (b) follows from Lemmas \ref{lem:H0} and \ref{lem:Hinf}, since
$H_0^{1,\log}$ is controlled by $H_0^{1-\varepsilon}$ for any fixed $\varepsilon > 0$ and, similarly,
$H_{\infty}^{-1,\log}$ is controlled by $H_{\infty}^{-1-\varepsilon}$.

To disprove the restricted weak type we will give counterexamples. We shall show that the weak type $(p,p)$
inequality fails either for $f_1 = \chi_{(1/2,1)}$ or for $f_2 = \chi_{(1,2)}$ (actually, here the only reason for taking
two different functions is simplicity of estimates that follow). Clearly, $f_1,f_2 \in L^p(x^{\delta}dx)$ for any $\delta$.

We have
$$
H_0^{1,\log}f_1(x) = \frac{1}{x^2}\int_{1/2}^1 \log\frac{x}y \,y \, dy \gtrsim \frac{\log x}{x^2}, \qquad x > 2,
$$
but the function $x \mapsto \chi_{(2,\infty)}(x) x^{-2}\log x$ does not belong to weak $L^p(x^{\delta}dx)$ when $\delta \ge 2p-1$.
Indeed, for sufficiently small $\lambda > 0$
$$
\lambda^{p} \int_{\{y > 2 : y^{-2}\log y > \lambda\}} x^{\delta}\, dx \ge \lambda^{p}
	\int_2^{\frac{\log^{1/2}\frac{1}{\sqrt{\lambda}}}{2\sqrt{\lambda}}} x^{\delta} \, dx \simeq
		\lambda^{p-(\delta+1)/2} \bigg( \log\frac{1}{\lambda}\bigg)^{(\delta+1)/2},
$$
and the last quantity is unbounded in $\lambda \to 0^+$ if $\delta \ge 2p-1$.
In the above we used the fact that asymptotically, as $\lambda \to 0^+$, the solution of $y^{-2}\log y = \lambda$
is $\frac{1}{\sqrt{\lambda}} \log^{1/2}\frac{1}{\sqrt{\lambda}}$.

For $H_{\infty}^{-1,\log}$ we write
$$
H_{\infty}^{-1,\log}f_2(x) = \frac{1}{x} \int_1^2 \log\frac{y}x\, dy \gtrsim \frac{1}x \log\frac{1}x, \qquad 0 < x < 1,
$$
and the function $x \mapsto \chi_{(0,1)}(x) \frac{1}x \log\frac{1}x$ is not in weak $L^p(x^{\delta}dx)$ when $\delta \le p-1$.
Indeed, for sufficiently large $\lambda$ one has
$$
\lambda^{p} \int_{\{y \in (0,1) : y^{-1}\log y^{-1} > \lambda \}} x^{\delta}\, dx \ge \lambda^p
	\int_0^{\frac{\log \lambda}{2\lambda}} x^{\delta}\, dx \simeq
		\begin{cases}
			\lambda^{p-\delta-1} \log^{\delta+1}\lambda, & \delta > -1, \\
			\infty, & \delta \le -1,
		\end{cases}
$$
and the last quantity is either infinite or unbounded in $\lambda \to \infty$ if $\delta \le p-1$.

The conclusion follows.
\end{proof}

Given $b > 0$, consider the following modifications of $H_0^{\xi}$ and $H_{\infty}^{\xi}$:
\begin{align*}
H_{0,b}^{\xi}f(x) & = x^{-\xi-1} \int_0^{x/b} f(y) y^{\xi}\, dy, \qquad x > 0, \\
H_{\infty,b}^{\xi}f(x) & = x^{\xi} \int_{b x}^{\infty} f(y) y^{-\xi-1}\, dy, \qquad x > 0.
\end{align*}
Logarithmic analogues $H_{0,b}^{1,\log}$ and $H_{\infty,b}^{-1,\log}$ are defined similarly.
The $b$-parametrized operators have the same mapping properties as their prototypes, see
Lemmas \ref{lem:H0}, \ref{lem:Hinf} and \ref{lem:Hlog}. Furthermore,
as the next result shows, they always differ from the prototypes by $L^p$-bounded operators.
\begin{pro} \label{prop:Hb}
Let $\xi,\delta \in \mathbb{R}$ and $1 \le p \le \infty$. Assume that $ b> 0$ is fixed. Then the operators
$$
H_{0}^{\xi}-H_{0,b}^{\xi}, \quad H_{\infty}^{\xi}-H_{\infty,b}^{\xi}, \quad
H_{0}^{1,\log} - H_{0,b}^{1,\log}, \quad H_{\infty}^{-1,\log}-H_{\infty,b}^{-1,\log}
$$
are bounded on $L^p(\mathbb{R}_+,x^{\delta}dx)$.
\end{pro}

\begin{proof}
See the arguments proving Lemma \ref{lem:NN} below, cf.\ \cite[p.\,125--126]{BHNV}.
\end{proof}

%%%%%%%%%%%%%%%%%%%%%%%%%%%%%%%%%%%%%%%%%%%%%%%%%%%%%%%%%%%%%%
\subsection{Auxiliary operators} \label{ssec:aux}

We will use the following auxiliary operators:
\begin{align*}
M^k_{\loc}f(x) & = \sup_{0 < u < x < v < ku} \frac{1}{v-u} \int_u^v |f(y)|\, dy, \qquad x > 0, \\
\mathcal{H}_{\nu,\loc}f(x) & = \frac{1}{\pi} \pv \int_{x/2}^{2x} \frac{(xy)^{-\nu-1/2}}{y-x} f(y)\, d\eta_{\nu}(y),
	\qquad x > 0, \\
\mathfrak{g}_{\nu,\loc}(f)(x) & = \bigg\| \int_{x/2}^{2x} (xy)^{-\nu-1/2}\frac{\partial}{\partial t}
	\mathcal{W}_t(x,y)f(y)\, d\eta_{\nu}(y) \bigg\|_{L^2(\mathbb{R}_+,tdt)}, \qquad x > 0, \\
T^{\xi}_{\psi}f(x) & = \sup_{t > 0}\bigg| x^{\xi} \int_x^{\infty} f(y) \psi(t,y) y^{-\xi-1}\, dy\bigg|, \qquad x > 0, \\
{N}f(x) & = \int_{x/2}^{2x} \frac{f(y)}{y}\, dy, \qquad x > 0, \\
N^{\log}f(x) & = \int_{x/2}^{2x} \frac{1}{y} \bigg( 1 + \log\frac{xy}{(x-y)^2} \bigg) f(y)\, dy, \qquad x > 0.
\end{align*}
Here $M_{\loc}^k$ is the local non-centered Hardy-Littlewood function (the parameter $k>1$), whereas
$\mathcal{H}_{\nu,\loc}$ and $\mathfrak{g}_{\nu,\loc}$
are extensions, in the parameter $\nu$, of the classical local Hilbert transform 
$\mathcal{H}_{-1/2,\loc}$ and the classical local vertical $g$-function $\mathfrak{g}_{-1/2,\loc}$.
In the $\mathfrak{g}_{\nu,\loc}$ formula $\mathcal{W}_t(x,y) = (4\pi t)^{-1/2}e^{-(x-y)^2/4t}$
is the classical Gauss-Weierstrass kernel.
The maximal operator $T_{\psi}^{\xi}$, for suitable $\psi$, is related to the dual Hardy
operator $H_{\infty}^{\xi}$. Notice that $|Nf(x)| \le N^{\log}|f|(x)$.

Mapping properties of these operators are essentially known, see \cite{NoSt,CRH,BHNV} and also references given there.
In particular, we have the following results.

\begin{lem} \label{lem:MH}
Let $k>1$ and $\nu,\delta \in \mathbb{R}$. Each of the operators $M_{\loc}^{k}$, $\mathcal{H}_{\nu,\loc}$
and $\mathfrak{g}_{\nu,\loc}$ is bounded on $L^p(\mathbb{R}_+,x^{\delta}dx)$, $1 < p < \infty$,
and from $L^1(\mathbb{R}_+,x^{\delta}dx)$ to weak $L^1(\mathbb{R}_+,x^{\delta}dx)$.
But none of these operators is bounded on $L^1(\mathbb{R}_+,x^{\delta}dx)$.
Moreover, $M_{\loc}^k$ is bounded on $L^{\infty}(\mathbb{R}_+)$.
\end{lem}

The next lemma restricted to $\xi > -1$ is contained in \cite[Lemma 3.3]{CRH} (the restriction $t\in (0,1)$ in the definition
of $T_{\psi}^{\xi}$ there is not relevant). The extension for $\xi \le 0$ in (a) follows from the pointwise control
$T_{\psi}^{\xi}f(x) \lesssim H_{\infty}^{\xi}|f|(x)$ and Lemma \ref{lem:Hinf}(a), while in (b) it is trivial. In case of (c)
the argument justifying \cite[Lemma 3.3(c)]{CRH} simply goes through. % for $\xi \le -1$.
More precisely, in view of the above control and Lemma \ref{lem:Hinf}(b),
we only need to treat the case $\delta = -\xi p -1$, $1<p<\infty$,
$\xi < 0$, and this is done as in \cite[p.\,336]{CRH}.

\begin{lem} \label{lem:Tpsi}
Let $\xi,\delta \in \mathbb{R}$ and $\psi(t,y) = (y^2/t)^{\epsilon}e^{-cy^2/t}$ with some $c,\epsilon >0$ fixed.
Consider $T_{\psi}^{\xi}$ on the measure space $(\mathbb{R}_+,x^{\delta}dx)$. Then
\begin{itemize}
\item[(a)] $T_{\psi}^{\xi}$ is of strong type $(p,p)$ when $1<p<\infty$ and $-\xi p -1 < \delta$;
\item[(b)] $T_{\psi}^{\xi}$ is of strong type $(\infty,\infty)$ if $\xi \ge 0$;
\item[(c)] $T_{\psi}^{\xi}$ is of weak type $(p,p)$ when $1 \le p < \infty$ and $-\xi p -1 \le \delta$, %($<$ in case $\xi=0$).
	with the last inequality strictened in case $\xi=0$.
\end{itemize}
\end{lem}

\begin{lem} \label{lem:NN}
Let $\delta \in \mathbb{R}$ and $1 \le p \le \infty$. Each of the operators $N$ and $N^{\log}$ is bounded on
$L^p(\mathbb{R}_+,x^{\delta}dx)$.
\end{lem}

\begin{rem} \label{rem:aux}
Let $b>1$ be fixed. Changing slightly the definitions of $\mathcal{H}_{\nu,\loc}$, $N$ and $N^{\log}$ so that the
integrations are over $(x/b,bx)$, and the definition of $T_{\psi}^{\xi}$ so that the integration is from $x/b$ or $bx$ to
$\infty$, does not affect the statements of Lemmas \ref{lem:MH}, \ref{lem:Tpsi} and \ref{lem:NN}.
\end{rem}

%%%%%%%%%%%%%%%%%%%%%%%%%%%%%%%%%%%%%%%%%%%%%%%%%%%%%%%%%%%%%%
\subsection{Boundedness transference principle} \label{ssec:trans}

Let $E \subset \mathbb{R}$ (think $E = (-1,\infty)$).
Assume that $\{K^{\nu} : \nu \in E\}$ is a family of operators acting on functions on $\mathbb{R}_+$.
Consider an associated family $\{\widetilde{K}^{\nu} : \nu \in -E\}$ given by
$$
\widetilde{K}^{\nu}f(x) = x^{-2\nu} K^{-\nu}\big(y^{2\nu}f\big)(x).
$$
Observe that for each $\nu \in E$ the following equivalence holds:
$K^{\nu}$ is an integral operator related to the measure space $(\mathbb{R}_+,d\eta_{\nu})$, i.e.\ it represents as
$$
K^{\nu}f(x) = \int_0^{\infty} K^{\nu}(x,y)f(y)\, d\eta_{\nu}(y),
$$
if and only if $\widetilde{K}^{-\nu}$ is an integral operator related to the measure space
$(\mathbb{R}_+,d\eta_{-\nu})$ with the corresponding integral kernel
$$
\widetilde{K}^{-\nu}(x,y) = (xy)^{2\nu} K^{\nu}(x,y).
$$
This remains true for principal value integral representations.

The result below describes the relation between two-weight $L^p-L^q$ boundedness of $K^{\nu}$ and $\widetilde{K}^{-\nu}$.
The proof is straightforward and thus omitted.
\begin{pro} \label{thm:trans}
Let $1 \le p,q \le \infty$ and $U,V$ be (power) weights on $\mathbb{R}_+$.
Fix an arbitrary $\nu \in -E$.
\begin{itemize}
\item[(i)]
$\widetilde{K}^{\nu}$ is well defined on $L^p(\mathbb{R}_+,U^p d\eta_{\nu})$ if and only if\\
$K^{-\nu}$ is well defined on $L^p(\mathbb{R}_+,(x^{2\nu (2/p-1)}U)^p d\eta_{-\nu})$.
\item[(ii)]
$\widetilde{K}^{\nu}$ is bounded from $L^p(\mathbb{R}_+,U^p d\eta_{\nu})$ to $L^q(\mathbb{R}_+,V^q d\eta_{\nu})$ if and only if \\
$K^{-\nu}$ is bounded from $L^p(\mathbb{R}_+,(x^{2\nu (2/p-1)}U)^p d\eta_{-\nu})$
to $L^q(\mathbb{R}_+,(x^{2\nu (2/q-1)}V)^q d\eta_{-\nu})$.
\end{itemize}
\end{pro}

No analogous relations are true in general for weak type or restricted weak type boundedness.

%%%%%%%%%%%%%%%%%%%%%%%%%%%%%%%%%%%%%%%%%%%%%%%%%%%%%%%%%%%%%%
%%%%%%%%%%%%%%%%%%%%%%%%%%%%%%%%%%%%%%%%%%%%%%%%%%%%%%%%%%%%%%

\section{Bessel semigroup maximal operator} \label{sec:max}

In this section we study the Bessel semigroup maximal operator in the exotic case.
It is well known that the integral kernel (with respect to $d\eta_{\nu}$) of $\{\exp(-t B_{\nu}^{\textrm{cls}})\}$ is
$$
W_t^{\nu}(x,y) = \frac{1}{2t} (xy)^{-\nu} e^{-(x^2+y^2)/4t} I_{\nu}\Big(\frac{xy}{2t}\Big), \qquad t,x,y > 0,
$$
and here $\nu > -1$ is in the classical range.
The integral kernel of $\{\exp(-t B_{\nu}^{\textrm{exo}})\}$, still with respect to $d\eta_{\nu}$, is
(see \cite[Section 4]{NSS})
\begin{equation} \label{exoWcls}
\widetilde{W}_t^{\nu}(x,y) = (xy)^{-2\nu} W_t^{-\nu}(x,y), \qquad t,x,y > 0.
\end{equation}
Here $\nu < 1$ is in the exotic range. Recall that for $\nu=0$ the classical and exotic settings coincide.

Denote
$$
W_t^{\nu}f(x) = \int_0^{\infty} W_t^{\nu}(x,y) f(y)\, d\eta_{\nu}(y), \qquad
\widetilde{W}_t^{\nu}f(x) = \int_0^{\infty} \widetilde{W}_t^{\nu}(x,y) f(y)\, d\eta_{\nu}(y), \qquad t,x > 0,
$$
and observe that
\begin{equation} \label{Wexotocls}
\widetilde{W}_t^{\nu}f(x) = x^{-2\nu} W_t^{-\nu}\big(y^{2\nu}f\big)(x), \qquad t,x > 0, \quad 0 \neq \nu < 1.
\end{equation}
We consider the corresponding maximal operators
$$
W_{*}^{\nu}f(x) = \sup_{t > 0} \big| W_t^{\nu}f(x)\big|, \qquad
\widetilde{W}_{*}^{\nu}f(x) = \sup_{t > 0} \big| \widetilde{W}_t^{\nu}f(x)\big|, \quad x > 0,
$$
on their natural domains, the first one for $\nu > -1$, the second one for $\nu < 1$.

In \cite[Theorem 2.1]{BHNV} the following mapping properties of $W_{*}^{\nu}$ were characterized:
strong type, weak type and restricted weak type $(p,p)$, $1\le p < \infty$, with respect to the measure space
$(\mathbb{R}_+,x^{\delta}dx)$.
\begin{thm}[{\cite{BHNV}}] \label{thm:maxWcls}
Let $\nu > -1$, $1 \le p < \infty$, $\delta \in \mathbb{R}$. Then the maximal operator $W_{*}^{\nu}$,
considered on the measure space $(\mathbb{R}_+,x^{\delta}dx)$, has the following mapping properties:
\begin{itemize}
\item[(a)]
$W_{*}^{\nu}$ is of strong type $(p,p)$ if and only if $p > 1$ and $-1 < \delta < (2\nu+2)p -1$;
\item[(b)]
$W_{*}^{\nu}$ is of weak type $(p,p)$ if and only if $-1 < \delta < (2\nu+2)p-1$, with the second inequality weakened
	in case $p=1$;
\item[(c)]
$W_{*}^{\nu}$ is of restricted weak type $(p,p)$ if and only if $-1 < \delta \le (2\nu+2)p-1$.
\end{itemize}
Moreover, $W_{*}^{\nu}$ is of strong type $(\infty,\infty)$.
\end{thm}

We will prove an analogous characterization in the exotic case, see Figure \ref{fig_1} below.
\begin{thm} \label{thm:maxWexo}
Let $0 \neq \nu < 1$, $1 \le p < \infty$, $\delta \in \mathbb{R}$. Then the maximal operator $\widetilde{W}_{*}^{\nu}$,
considered on the measure space $(\mathbb{R}_+,x^{\delta}dx)$, has the following mapping properties:
\begin{itemize}
\item[(a)]
$\widetilde{W}_{*}^{\nu}$ is of strong type $(p,p)$ if and only if $p > 1$ and $2\nu p - 1 < \delta < 2p -1$;
\item[(b)]
$\widetilde{W}_{*}^{\nu}$ is of weak type $(p,p)$ if and only if $2\nu p -1 \le \delta < 2p-1$,
	with the second inequality weakened in case $p=1$;
\item[(c)]
$\widetilde{W}_{*}^{\nu}$ is of restricted weak type $(p,p)$ if and only if $2\nu p-1 \le \delta \le 2p-1$.
\end{itemize}
Moreover, $\widetilde{W}_{*}^{\nu}$ is of strong type $(\infty,\infty)$ if and only if $\nu < 0$.
\end{thm}

%%%%% PICTURE I

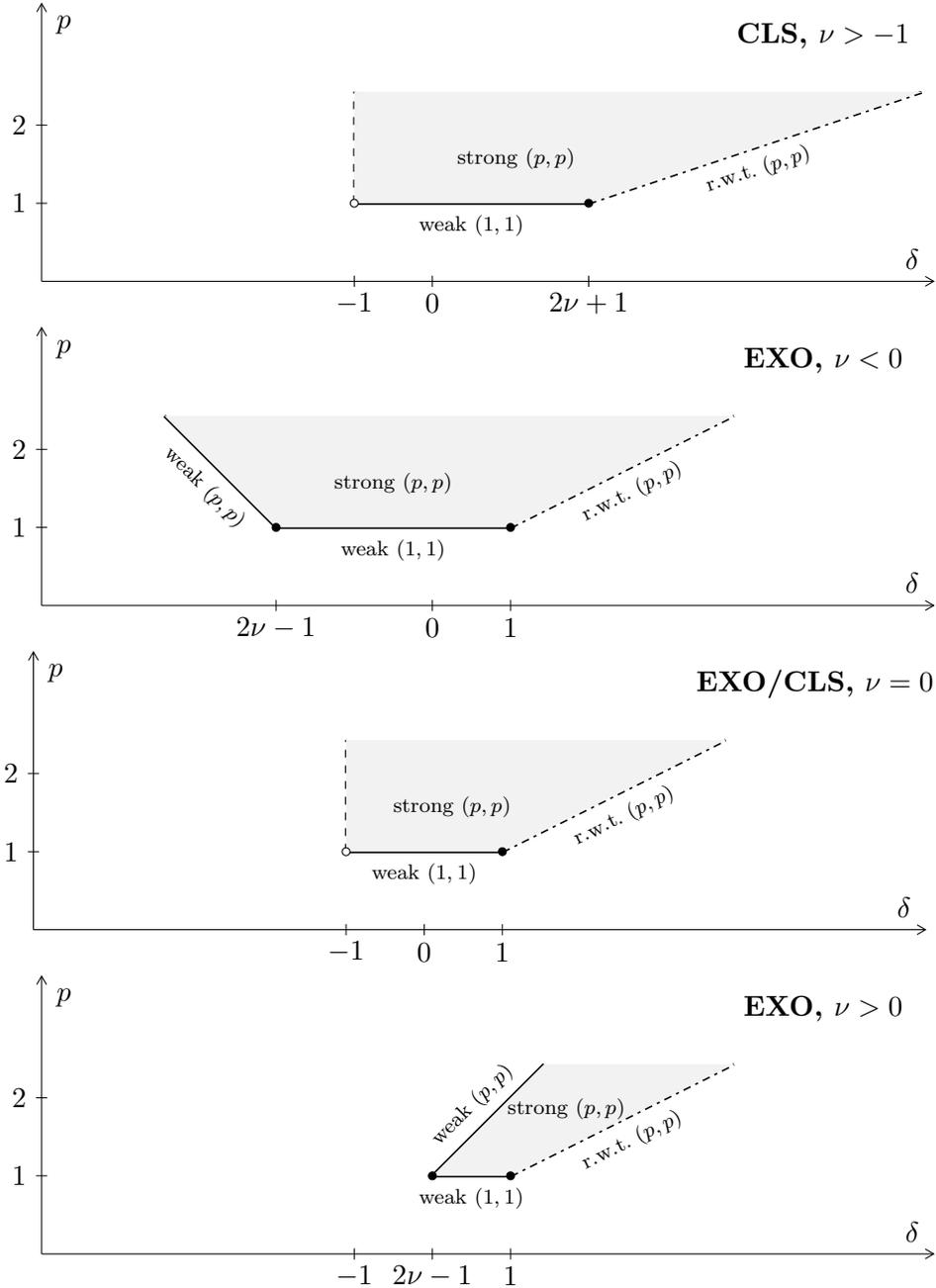
\begin{figure}
\centering
\begin{tikzpicture}[scale=1.5]

\draw[arrows=-angle 60] (0,0) -- (0,2.5);
\draw[arrows=-angle 60] (0,0) -- (8,0);

\node at (0.2,2.3) {$p$};
\node at (7.8,0.2) {$\delta$};

\draw[very thin] (2.8,-0.05) -- (2.8,0.05);
\node at (2.8,-0.2) {$-1$};

\draw[very thin] (3.5,-0.05) -- (3.5,0.05);
\node at (3.5,-0.2) {$0$};

\draw[very thin] (4.9,-0.05) -- (4.9,0.05);
\node at (4.9,-0.2) {$2\nu+1$};

\draw[very thin] (-0.05,0.7) -- (0.05,0.7);
\node at (-0.2,0.7) {$1$};

\draw[very thin] (-0.05,1.4) -- (0.05,1.4);
\node at (-0.2,1.4) {$2$};

\draw[very thick, dash dot] (4.9,0.7) -- (7.9,1.7);
\draw[very thick] (2.8,0.7) -- (4.9,0.7);
\draw[thick, dashed] (2.8,0.7) -- (2.8,1.7);

\fill[black!5!white] (2.8,1.7) -- (2.8,0.7) -- (4.9,0.7) -- (7.9,1.7);

\node at (3.85,0.5) {\scriptsize weak $(1,1)$};
\node at (4.25,1.1) {\scriptsize strong $(p,p)$};
\node[rotate=18.43] at (6.42, 1.002) {\scriptsize r.w.t.\ $(p,p)$};

\node at (7,2.2) {\textbf{CLS, $\nu>-1$}};

\filldraw[fill=white] (2.8,0.7) circle(1pt);
\filldraw[fill=black] (4.9,0.7) circle(1pt);

\end{tikzpicture} 

%%%%% PICTURE II

\begin{tikzpicture}[scale=1.5]

\draw[arrows=-angle 60] (0,0) -- (0,2.5);
\draw[arrows=-angle 60] (0,0) -- (8,0);

\node at (0.2,2.3) {$p$};
\node at (7.8,0.2) {$\delta$};

\draw[very thin] (2.1,-0.05) -- (2.1,0.05);
\node at (2.1,-0.2) {$2\nu-1$};

\draw[very thin] (3.5,-0.05) -- (3.5,0.05);
\node at (3.5,-0.2) {$0$};

\draw[very thin] (4.2,-0.05) -- (4.2,0.05);
\node at (4.2,-0.2) {$1$};
\draw[very thin] (-0.05,0.7) -- (0.05,0.7);
\node at (-0.2,0.7) {$1$};

\draw[very thin] (-0.05,1.4) -- (0.05,1.4);
\node at (-0.2,1.4) {$2$};

\draw[very thick, dash dot] (4.2,0.7) -- (6.2,1.7);
\draw[very thick] (2.1,0.7) -- (4.2,0.7);
\draw[very thick] (2.1,0.7) -- (1.1,1.7);

\fill[black!5!white] (1.1,1.7) -- (2.1,0.7) -- (4.2,0.7) -- (6.2,1.7);

\node at (3.15,0.5) {\scriptsize weak $(1,1)$};
\node at (3.15,1.1) {\scriptsize strong $(p,p)$};
\node[rotate=-45] at (1.46, 1.06) {\scriptsize weak $(p,p)$};
\node[rotate=26.56] at (5.29, 1.02) {\scriptsize r.w.t.\ $(p,p)$};

\node at (7,2.2) {\textbf{EXO, $\nu<0$}};

\filldraw[fill=black] (2.1,0.7) circle(1pt);
\filldraw[fill=black] (4.2,0.7) circle(1pt);

\end{tikzpicture} 

%%%%% PICTURE III

\begin{tikzpicture}[scale=1.5]

\draw[arrows=-angle 60] (0,0) -- (0,2.5);
\draw[arrows=-angle 60] (0,0) -- (8,0);

\node at (0.2,2.3) {$p$};
\node at (7.8,0.2) {$\delta$};

\draw[very thin] (2.8,-0.05) -- (2.8,0.05);
\node at (2.8,-0.2) {$-1$};

\draw[very thin] (3.5,-0.05) -- (3.5,0.05);
\node at (3.5,-0.2) {$0$};

\draw[very thin] (4.2,-0.05) -- (4.2,0.05);
\node at (4.2,-0.2) {$1$};

\draw[very thin] (-0.05,0.7) -- (0.05,0.7);
\node at (-0.2,0.7) {$1$};

\draw[very thin] (-0.05,1.4) -- (0.05,1.4);
\node at (-0.2,1.4) {$2$};

\draw[very thick, dash dot] (4.2,0.7) -- (6.2,1.7);
\draw[very thick] (2.8,0.7) -- (4.2,0.7);
\draw[thick,dashed] (2.8,0.7) -- (2.8,1.7);

\fill[black!5!white] (2.8,1.7) -- (2.8,0.7) -- (4.2,0.7) -- (6.2,1.7);

\node at (3.5,0.5) {\scriptsize weak $(1,1)$};
\node at (3.75,1.1) {\scriptsize strong $(p,p)$};
\node[rotate=26.56] at (5.29, 1.02) {\scriptsize r.w.t.\ $(p,p)$};

\node at (7,2.2) {\textbf{EXO/CLS, $\nu=0$}};

\filldraw[fill=white] (2.8,0.7) circle(1pt);
\filldraw[fill=black] (4.2,0.7) circle(1pt);

\end{tikzpicture} 

%%%%% PICTURE IV

\begin{tikzpicture}[scale=1.5]

\draw[arrows=-angle 60] (0,0) -- (0,2.5);
\draw[arrows=-angle 60] (0,0) -- (8,0);

\node at (0.2,2.3) {$p$};
\node at (7.8,0.2) {$\delta$};

\draw[very thin] (2.8,-0.05) -- (2.8,0.05);
\node at (2.8,-0.2) {$-1$};

\draw[very thin] (3.5,-0.05) -- (3.5,0.05);
\node at (3.5,-0.2) {$2\nu-1$};

\draw[very thin] (4.2,-0.05) -- (4.2,0.05);
\node at (4.2,-0.2) {$1$};
\draw[very thin] (-0.05,0.7) -- (0.05,0.7);
\node at (-0.2,0.7) {$1$};

\draw[very thin] (-0.05,1.4) -- (0.05,1.4);
\node at (-0.2,1.4) {$2$};

\draw[very thick, dash dot] (4.2,0.7) -- (6.2,1.7);
\draw[very thick] (3.5,0.7) -- (4.2,0.7);
\draw[very thick] (3.5,0.7) -- (4.5,1.7);

\fill[black!5!white] (4.5,1.7) -- (3.5,0.7) -- (4.2,0.7) -- (6.2,1.7);

\node at (3.85,0.5) {\scriptsize weak $(1,1)$};
\node at (4.7,1.3) {\scriptsize strong $(p,p)$};
\node[rotate=45] at (3.86, 1.34) {\scriptsize weak $(p,p)$};
\node[rotate=26.56] at (5.29, 1.02) {\scriptsize r.w.t.\ $(p,p)$};

\node at (7,2.2) {\textbf{EXO, $\nu>0$}};

\filldraw[fill=black] (3.5,0.7) circle(1pt);
\filldraw[fill=black] (4.2,0.7) circle(1pt);

\end{tikzpicture}

\caption{Visualization of Theorems \ref{thm:maxWcls} and \ref{thm:maxWexo}, with the choices
$\nu =1/2$, $\nu=-1/2$, $\nu=0$ and $\nu=1/2$, respectively.
Note the phase shift occurring at $\nu=0$ in the exotic case.} \label{fig_1}

\end{figure}

Choosing $\delta=2\nu+1$, so that $x^{\delta}dx = d\eta_{\nu}$ is the natural weight, we see that for $\nu \le 0$ the operator
$\widetilde{W}_{*}^{\nu}$ behaves in a standard way, i.e.\ it is strong type $(p,p)$ for $p>1$ and of weak type $(1,1)$.
Actually, this was proved already in \cite{NSS}, even in a multi-dimensional situation.
The same standard behavior is shared by $W^{\nu}_{*}$ for the full classical range of the parameter $\nu > -1$.
However, when the exotic range is $\nu \in (0,1)$ the operator $\widetilde{W}_{*}^{\nu}$ is strong type $(p,p)$ if and only if
$\nu+1 < p < 1 + 1/\nu$, with restricted weak type $(p,p)$ and weak type $(p,p)$ at the endpoints, respectively.
This is an instance of the so-called pencil phenomenon, see e.g.\ \cite{MST,NoSj1}.

The power weights admitted in Theorem \ref{thm:maxWcls} are exactly power $A_p$ weights associated with the space
of homogeneous type $(\mathbb{R}_+,|\cdot|,d\eta_{\nu})$. But this is no longer true in the context of Theorem
\ref{thm:maxWexo}. Moreover, when $\nu \le -1$ the $A_p$ class in this context is not even defined.

The proof of Theorem \ref{thm:maxWexo} relies on a standard strategy used in \cite{BHNV} and other papers.
Sufficiency of the conditions will be obtained by establishing a control of $\widetilde{W}_{*}^{\nu}$ in terms of several
special operators whose mapping properties are known, see Section \ref{sec:tech}.
Necessity will be shown by constructing suitable counterexamples.

\begin{proof}[{Proof of Theorem \ref{thm:maxWexo}, sufficiency part.}]
Recall that $0 \neq \nu < 1$.
We shall use the following estimate of $\widetilde{W}_t^{\nu}(x,y)$,
which is a consequence of the basic Bessel function asymptotics \eqref{bes:est}.
It can also be deduced from \eqref{exoWcls} and \cite[Lemma 3.1]{BHNV}.
For each $\nu < 1$, there exists $c>0$ such that
$$
\widetilde{W}_t^{\nu}(x,y) \lesssim 
		\begin{cases}
			x^{-2} y^{-2\nu}, & y \le x/2, \\
			(xy)^{-\nu-1} + t^{-1/2}(xy)^{-\nu-1/2} e^{(x-y)^2/4t}, & x/2 < y < 2x, \\
			x^{-2\nu} y^{-2} (y^2/t)^{-\nu+1} e^{-c y^2/t}, & 2x \le y,
		\end{cases}
$$
uniformly in $t,x,y > 0$. Consequently, $\widetilde{W}_*^{\nu}$ can be controlled by our special operators
(see Sections \ref{ssec:Hardy} and \ref{ssec:aux}), namely
$$
\widetilde{W}_{*}^{\nu}f(x) \lesssim H_0^1|f|(x) + M^4_{\loc}f(x) + T_{\psi}^{-2\nu}|f|(x),
$$
with $\psi(t,y) = (y^2/t)^{-\nu+1}e^{-cy^2/t}$.

Appealing now to Lemmas \ref{lem:H0}, \ref{lem:MH} and \ref{lem:Tpsi} we conclude the following mapping properties of
the operator $\widetilde{W}_{*}^{\nu}$, considered on the measure space $(\mathbb{R}_+,x^{\delta}dx)$.
For $1<p<\infty$, $\widetilde{W}_{*}^{\nu}$ is of strong type $(p,p)$ if $2\nu p -1 < \delta < 2p-1$.
Moreover, $\widetilde{W}_{*}^{\nu}$ is of strong type $(\infty,\infty)$ if $\nu < 0$. 
When $1< p < \infty$ and $2\nu p -1 \le \delta < 2p -1$, $\widetilde{W}_{*}^{\nu}$ is weak type $(p,p)$.
If $2\nu -1 \le \delta \le 1$, then $\widetilde{W}_{*}^{\nu}$ is of weak type $(1,1)$.
Finally, if $1< p < \infty$ and $2\nu p -1 \le \delta \le 2p -1$,
then $\widetilde{W}_{*}^{\nu}$ is of restricted weak type $(p,p)$. Altogether, this gives the sufficiency part in
Theorem \ref{thm:maxWexo}.
\end{proof}

\begin{proof}[{Proof of Theorem \ref{thm:maxWexo}, necessity part.}]
We will construct suitable counterexamples to prove the following statements (as before, we assume $0\neq \nu < 1$
and $1\le p < \infty$; the underlying space is always $(\mathbb{R}_+,x^{\delta}dx)$).
\begin{itemize}
\item[(A)] If $\widetilde{W}_{*}^{\nu}$ is of restricted weak type $(p,p)$ then $2\nu p -1 \le \delta \le 2p -1$.
\item[(B)] $\widetilde{W}_{*}^{\nu}$ is not of weak type $(p,p)$ when $p>1$ and $\delta = 2p-1$.
\item[(C)] $\widetilde{W}_{*}^{\nu}$ is not of strong type $(1,1)$ if $2\nu-1 \le \delta \le 1$.
\item[(D)] $\widetilde{W}_{*}^{\nu}$ is not of strong type $(p,p)$ when $p>1$ and $\delta = 2\nu p -1$.
\item[(E)] $\widetilde{W}_{*}^{\nu}$ is not of strong type $(\infty,\infty)$ when $\nu > 0$.
\end{itemize}
To this end, we always consider $f \ge 0$. Note that (A)--(E) altogether give the desired conclusion.

We begin with proving (A). First, observe that the 
standard asymptotics \eqref{bes:est} for the Bessel function lead to the following lower bounds for the kernel:
\begin{align*}
\widetilde{W}_t^{\nu}(x,y) & \gtrsim t^{\nu-1} (xy)^{-2\nu} e^{-(x^2+y^2)/4t}, \qquad xy < t, \\
\widetilde{W}_t^{\nu}(x,y) & \gtrsim t^{-1/2} (xy)^{-\nu-1/2} e^{-(x-y)^2/4t}, \qquad xy > t.
\end{align*}
Consequently, we get the lower bounds
\begin{align} \label{3exo}
\widetilde{W}_t^{\nu}f(x) & \gtrsim t^{\nu-1} x^{-2\nu} \int_0^{t/x} y e^{-(x^2+y^2)/4t}f(y)\, dy, \qquad t,x > 0,\\
\widetilde{W}_t^{\nu}f(x) & \gtrsim t^{-1/2} x^{-\nu-1/2} \int_{t/x}^{\infty} y^{\nu+1/2} e^{-(x-y)^2/4t} f(y)\, dy,
	\qquad t,x > 0. \label{4exo}
\end{align}

Taking $t=x^2$ in \eqref{3exo} we infer that
\begin{equation} \label{5exo}
\widetilde{W}_{*}^{\nu}f(x) \gtrsim x^{-2} \int_0^{x} y f(y)\, dy, \qquad x > 0.
\end{equation}
Consider now the specific $f= \chi_{(1,2)}$. The above implies
\begin{equation} \label{55exo}
\widetilde{W}_{*}^{\nu}f(x) \gtrsim x^{-2}, \qquad x > 2.
\end{equation}
Therefore, if $\widetilde{W}_{*}^{\nu}$ satisfies the restricted weak type $(p,p)$ inequality, then
$$
\lambda^p \int_{\{y> 2: y^{-2}> \lambda\}} x^{\delta}\, dy \lesssim 1, \qquad \lambda > 0.
$$
Taking into account $\lambda>0$ small, this implies that the function $\lambda \mapsto \lambda^{-(\delta+1)/2+p}$ is bounded
in a neighborhood of $0$, hence necessarily $-(\delta+1)/2+p \ge 0$, i.e.\ $\delta \le 2p -1$.
This gives the upper bound for $\delta$ in (A).

On the other hand, taking $t=2$ in \eqref{3exo} we get
$$
\widetilde{W}_{*}^{\nu}f(x) \gtrsim x^{-2\nu} \int_0^2 y f(y)\, dy, \qquad 0 < x <1.
$$
Thus, with $f=\chi_{(1,2)}$ chosen as before,
\begin{equation} \label{6exo}
\widetilde{W}_{*}^{\nu}f(x) \gtrsim x^{-2\nu}, \qquad 0 < x < 1.
\end{equation}
The restricted weak type $(p,p)$ inequality for $\widetilde{W}_{*}^{\nu}$ then implies
\begin{equation} \label{9exo}
\lambda^p \int_{\{0 < y < 1: y^{-2\nu} > \lambda\}} x^{\delta} \, dx \lesssim 1, \qquad \lambda > 0.
\end{equation}
If $\nu < 0$ and $\delta < -1$ (note that in this case $2\nu p -1 < -1$), and $\lambda>0$ is small, this leads to
$$
\lambda^p \int_{\lambda^{-1/2\nu}}^1 x^{\delta} \, dx \simeq \lambda^{-(\delta+1)/2\nu+p} \lesssim 1.
$$
Letting $\lambda \to 0^+$, we conclude that $-(\delta+1)/2\nu + p\ge 0$, i.e.\ $2\nu p -1 \le \delta$.
If $\nu > 0$ (in this case $2\nu p -1 > -1$), for $\lambda$ large \eqref{9exo} implies
$$
1 \gtrsim \lambda^p \int_0^{\lambda^{-1/2\nu}} x^{\delta}\, dx \simeq 
 \begin{cases}
	\lambda^{-(\delta+1)/2\nu+p}, & \delta > -1, \\
	\infty, & \delta \le -1.
 \end{cases}
$$
Letting $\lambda \to \infty$, we obtain $-(\delta+1)/2\nu+p \le 0$, i.e.\ $2\nu p -1 \le \delta$.
The lower bound for $\delta$ in (A) follows.$^{\ddagger}$
\footnotetext{${\ddagger}$ The corresponding counterexample in the proof of \cite[Theorem 2.1]{BHNV} is wrong,
since the second bound on p.\ 114, line 6 there does not hold because of the exponential factor.
Fortunately, the problem can easily be fixed by using the bound $W_2^{\nu}(x,y) \gtrsim 1$, $x,y \in (0,2)$, and
as a consequence $W_{*}^{\nu}f(x) \gtrsim 1$, $x \in (0,1)$, with $f=\chi_{(1,2)}$.
The correct counterexample was pointed out by J.J.\ Betancor.
}

Observe also that \eqref{6exo} readily implies (E). Moreover, it implies (D) as well, since the necessary condition for
the function $x \mapsto \chi_{(0,1)}(x) x^{-2\nu}$ to be in $L^p(x^{\delta}dx)$ is $\delta > 2\nu p -1$.

Next we show (B). Let $\delta = 2p-1$ for some $1<p<\infty$. Suppose on the contrary that $\widetilde{W}_{*}^{\nu}$
is weak type $(p,p)$ with respect to $(\mathbb{R}_+,x^{\delta}dx)$. It is straightforward to check that the function
$x \mapsto \chi_{(0,1)}(x) x^{1-\delta}$ is not in $L^{p'}(x^{\delta}dx)$,
hence it does not define a bounded functional on $L^p(x^{\delta}dx)$ via the standard pairing.
Therefore, there exists a sequence of non-negative functions $f_n$, $n=1,2,\ldots$, in $L^p(\mathbb{R}_+,x^{\delta}dx)$ such that
$\|f_n\|_{L^p(x^{\delta}dx)} \le 1$ and $\int_0^1 f_n(x) x \, dx \ge n$. Further, by \eqref{5exo},
$$
\widetilde{W}_{*}^{\nu}f_n(x) \gtrsim n x^{-2}, \qquad n=1,2,\ldots, \quad x > 1.
$$
Now we see that the supposed weak type $(p,p)$ estimate for $\widetilde{W}_{*}^{\nu}$ implies
$$
\int_{\{y>1 : n y^{-2} > n-1\}} x^{\delta} \, dx \lesssim \frac{\|f_n\|^p_{L^p(x^{\delta}dx)}}{(n-1)^p} \le \frac{1}{(n-1)^p},
\qquad n \ge 2.
$$
Since the integral above is comparable to $n^p/(n-1)^p-1$, this means that the sequence
$n^p-(n-1)^p$ is bounded, which is a contradiction because $p>1$.

Considering (C), we actually give a counterexample showing that $\widetilde{W}_{*}^{\nu}$ is not bounded
on $L^1(x^{\delta}dx)$ for any $\delta \in \mathbb{R}$. Restricting the interval
of integration in \eqref{4exo} we get
$$
\widetilde{W}_t^{\nu}f(x) \gtrsim t^{-1/2} \int_{x/2}^{2x} e^{-(x-y)^2/4t} f(y)\, dy, \qquad t < 1/2, \quad x> 1.
$$
From here we argue as in \cite[p.\,114--115]{BHNV}. For $0 < \epsilon < 1/2$, consider $1 < y < 1+\epsilon$
and $1+2\epsilon< x < 2$. Then $x/2 < y < 2x$. Moreover, choosing $t=(x-1)^2/2 < 1/2$ we have
$(x-y)^2/4t \lesssim 1$. Taking $f_{\epsilon} = \chi_{(1,1+\epsilon)}$, we arrive at the bound
$$
\widetilde{W}_{*}^{\nu}f_{\epsilon}(x) \gtrsim \frac{\epsilon}{x-1}, \qquad 1+2\epsilon < x < 2,
$$
uniformly in $\epsilon$. This leads to the norm estimate
$$
\big\| \widetilde{W}_{*}^{\nu}f_{\epsilon} \big\|_{L^1(x^{\delta}dx)} \gtrsim \epsilon \int_{1+2\epsilon}^{2}
	\frac{x^{\delta}\, dx}{x-1} \simeq \epsilon \log\frac{1}{2\epsilon}.
$$
On the other hand, $\|f_{\epsilon}\|_{L^1(x^{\delta}dx)} \simeq \epsilon$. Letting $\epsilon \to 0^+$
we get the conclusion.

Proving the necessity part in Theorem \ref{thm:maxWexo} is complete.
\end{proof}

\begin{rem}
Strong type results in Theorem \ref{thm:maxWexo}(a) can be transferred more directly from Theorem \ref{thm:maxWcls}(a) by means
of Proposition \ref{thm:trans}. But no similar simple transference exists in case of weak and restricted weak type
estimates.
\end{rem}

Let $\widetilde{P}_t^{\nu}$ and $\widetilde{P}_{*}^{\nu}$ be the Bessel-Poisson semigroup
$\{\exp(-t(B_{\nu}^{\textrm{exo}})^{1/2})\}$ analogues of
$\widetilde{W}_t^{\nu}$ and $\widetilde{W}_{*}^{\nu}$, and similarly for the classical (untilded) objects.
In view of \eqref{exoWcls} and the subordination principle,
\begin{equation} \label{exoPcls}
\widetilde{P}_t^{\nu}(x,y) = (xy)^{-2\nu} P_t^{-\nu}(x,y), \qquad t,x,y > 0.
\end{equation}

We have the following bounds for $\widetilde{P}_t^{\nu}(x,y)$ and $\widetilde{P}_{*}^{\nu}$.
\begin{pro}
Let $\nu < 1$. Then, for any $C>1$ fixed,
\begin{align*}
\widetilde{P}_{t}^{\nu}(x,y) & \gtrsim \frac{t}{t^2+(x-y)^2}, \qquad 1<x,y<2, \quad t \le C, \\
\widetilde{P}_{t}^{\nu}(x,y) & \gtrsim \frac{(xy)^{-2\nu}}{(1+y^2)^{-\nu+3/2}},
	\qquad 0 < x < 1, \quad C^{-1} \le t \le C,\\
\widetilde{P}_{t}^{\nu}(x,y) & \gtrsim \chi_{\{y < x\}} x^{-2} y^{-2\nu}, \qquad x,y > 0, \quad C^{-1}x \le t \le Cx.
\end{align*}
\end{pro}

\begin{proof}
Combine \eqref{exoPcls} with \cite[Proposition 6.1]{BHNV}.
\end{proof}

\begin{cor} \label{cor:P}
Let $\nu < 1$ be fixed. Then, uniformly in $f \ge 0$,
\begin{align*}
\widetilde{P}_{*}^{\nu}f(x) & \gtrsim \frac{1}{x-1} \int_1^2 \frac{f(y)\, dy}{1+ \big(\frac{x-y}{x-1}\big)^2},
	\qquad 1 < x < 2,\\
\widetilde{P}_{*}^{\nu}f(x) & \gtrsim x^{-2\nu} \int_0^{\infty} \frac{y f(y)\, dy}{(1+y^2)^{-\nu+3/2}}, \qquad 0 < x < 1,\\
\widetilde{P}_{*}^{\nu}f(x) & \gtrsim x^{-2} \int_0^x y f(y)\, dy, \qquad x > 0.
\end{align*} 
\end{cor}

It was shown in \cite[Theorem 2.2]{BHNV} that Theorem \ref{thm:maxWcls} holds if $W_{*}^{\nu}$ is replaced
by $P_{*}^{\nu}$. A similar statement is true for the exotic counterparts.
\begin{pro} \label{prop:Poisson}
Statement of Theorem \ref{thm:maxWexo} remains true if $\widetilde{W}_{*}^{\nu}$ is replaced by $\widetilde{P}_{*}^{\nu}$.
\end{pro}
Indeed, the sufficiency part is then easily concluded by Theorem \ref{thm:maxWexo} and the subordination principle.
The necessity part is less obvious; it follows by the estimates from Corollary \ref{cor:P}
and exactly the same counterexamples as those already presented in the proof of Theorem \ref{thm:maxWexo}.
We leave the details to interested readers.

%%%%%%%%%%%%%%%%%%%%%%%%%%%%%%%%%%%%%%%%%%%%%%%%%%%%%%%%%%%%%%%
%%%%%%%%%%%%%%%%%%%%%%%%%%%%%%%%%%%%%%%%%%%%%%%%%%%%%%%%%%%%%%%

\section{Riesz transforms} \label{sec:riesz}

Recall (see e.g.\ \cite{BHNV}) that the Riesz transform associated with the classical Bessel operator $B_{\nu}^{\textrm{cls}}$
is formally defined by
\begin{equation} \label{R_fd}
R_{\nu} = D \big( B_{\nu}^{\textrm{cls}} \big)^{-1/2}, \qquad \nu > -1.
\end{equation}
Here $D = \frac{d}{dx}$ is the standard derivative. Its appearance is justified by the following factorization of the Bessel
differential operator:
\begin{equation} \label{B_dec}
B_{\nu} = D_{\nu}^{*} D, \qquad \textrm{where} \quad D_{\nu}^{*} = - x^{-2\nu-1}\frac{d}{dx} x^{2\nu+1}.
\end{equation}
Note that $D_{\nu}^{*}$ is the formal adjoint of $D$ in $L^2(d\eta_{\nu})$.

However, a rigorous definition of $R_{\nu}$ is not a straightforward issue, see e.g.\ \cite[Section 4]{BHNV}.
One considers test functions $f \in C_c^{\infty}(\mathbb{R}_+)$ and first defines the compensated potential operator
$$
\big( B_{\nu}^{\textrm{cls}}\big)^{-1/2} f(x) = \frac{1}{\sqrt{\pi}} \int_0^{\infty} \Big( W_t^{\nu}f(x)
	- \chi_{\{\nu \le -1/2\}} W_t^{\nu}f(0) \Big) \, \frac{dt}{\sqrt{t}}.
$$
Here $W_t^{\nu}f(0) = \lim_{x \to 0^+} W_t^{\nu}f(x)$ indeed exists for each $t > 0$.
The compensating term for $\nu \le -1/2$ is necessary since without it the integral would diverge.
One shows that the function $(B_{\nu}^{\textrm{cls}})^{-1/2}f$ is well defined and differentiable.
Then $R_{\nu}$ is defined strictly by \eqref{R_fd}. 

It is known \cite[Proposition 4.2]{BHNV} that for $f \in C_c^{\infty}(\mathbb{R}_+)$ the Riesz operator has
a singular integral representation
$$
R_{\nu}f(x) = \pv \int_0^{\infty} R_{\nu}(x,y) f(y)\, d\eta_{\nu}(y), \qquad x > 0,
$$
where the kernel is given by
$$
R_{\nu}(x,y) = \frac{1}{\sqrt{\pi}} \int_0^{\infty} \frac{\partial}{\partial x} W_t^{\nu}(x,y)\, \frac{dt}{\sqrt{t}},
	\qquad x,y > 0, \quad x \neq y.
$$

The Riesz transform $R_{\nu}$ extends uniquely to a bounded operator on $L^2(d\eta_{\nu})$.
This extension can be represented in terms of the (modified) Hankel transform, see \cite[Section 2]{BHNV}, and the same
is true for $R_{\nu}^{*}$, the adjoint of $R_{\nu}$ in $L^2(d\eta_{\nu})$. More precisely,
\begin{align*}
R_{\nu}f(x) & = - x h_{\nu+1} \big( y^{-1}h_{\nu}(y)\big)(x), \qquad f \in L^2(d\eta_{\nu}), \\
R_{\nu}^{*}f(x) & = -h_{\nu}\big( y h_{\nu+1}\big(z^{-1}f(z)\big)(y)\big)(x), \qquad f \in L^2(d\eta_{\nu}),
\end{align*}
where
$$
h_{\nu}f(z) = \int_0^{\infty} f(x) \varphi_z^{\nu}(x)\, d\eta_{\nu}(x), \qquad z > 0, \quad \nu > -1,
$$
is the modified Hankel transform, being $\varphi_z^{\nu}(x) = (xz)^{-\nu}J_{\nu}(xz)$ with $J_{\nu}$ denoting
the (oscillating) Bessel function of the first kind and order $\nu$, cf.\ e.g.\ \cite{Leb,lib,handbook,Wat}.
Note the identities
$$
R_{\nu}^{*} R_{\nu} = R_{\nu} R_{\nu}^{*} = \textrm{Id}
$$
that hold on $L^2(d\eta_{\nu})$. Note also that $R_{\nu}^{*}$ admits a singular integral representation on
$C_c^{\infty}(\mathbb{R}_+)$, the kernel being $R_{\nu}^{*}(x,y) = R_{\nu}(y,x)$.

Now we pass to the Riesz transform associated with the exotic Bessel operator $B_{\nu}^{\textrm{exo}}$.
Similarly as in the classical case, taking into account \eqref{B_dec} we formally define
\begin{equation} \label{R_fde}
\widetilde{R}_{\nu} = D \big( B_{\nu}^{\textrm{exo}}\big)^{-1/2}, \qquad \nu < 1.
\end{equation}
To make this rigorous for, say, $f \in C_c^{\infty}(\mathbb{R}_+)$, we define suitably the compensated potential operator
$(B_{\nu}^{\textrm{exo}})^{-1/2}$. Let
$$
\big( B_{\nu}^{\textrm{exo}}\big)^{-1/2} f(x) = \frac{1}{\sqrt{\pi}} \int_0^{\infty} \Big(
	\widetilde{W}_t^{\nu}f(x) - \chi_{\{\nu \ge 1/2\}} x^{-2\nu} W_t^{-\nu}\big(y^{2\nu}f(y)\big)(0) \Big)\, \frac{dt}{\sqrt{t}}.
$$
Here the compensating term $x^{-2\nu} W_t^{-\nu}(y^{2\nu}f)(0)$ is needed in case $\nu \ge 1/2$ to make the integral convergent.
Note that the seemingly more suitable term $\widetilde{W}_t^{\nu}f(0) = \lim_{x \to 0^{+}} \widetilde{W}_t^{\nu}f(x)$
cannot be used for this purpose.

It follows that
\begin{align} \nonumber
\big( B_{\nu}^{\textrm{exo}}\big)^{-1/2}f(x) & = x^{-2\nu} \frac{1}{\sqrt{\pi}} \int_0^{\infty} \Big(
	W_t^{-\nu}\big( y^{2\nu}f(y) \big)(x) - \chi_{\{\nu \ge 1/2\}} W_t^{-\nu}\big( y^{2\nu}f(y)\big)(0) \Big)\, \frac{dt}{\sqrt{t}}\\
& = x^{-2\nu} \big( B_{-\nu}^{\textrm{cls}} \big)^{-1/2}\big(y^{2\nu}f(y)\big)(x). \label{Bpe}
\end{align}
Thus $(B_{\nu}^{\textrm{exo}})^{-1/2}f$ is well defined and differentiable for $f \in C_c^{\infty}(\mathbb{R}_+)$, and now
\eqref{R_fde} has a rigorous meaning. Moreover, by \eqref{Bpe},
\begin{equation} \label{37}
\widetilde{R}_{\nu}f(x) = x^{-2\nu} R_{-\nu}\big( y^{2\nu}f(y)\big)(x)
	- \frac{2\nu}x x^{-2\nu}\big(B_{-\nu}^{\textrm{cls}}\big)^{-1/2}\big(y^{2\nu}f(y)\big)(x),
\end{equation}
where $(B_{\nu}^{\textrm{cls}})^{-1/2}$ is the classical compensated potential operator.

We have the singular integral representation
$$
\widetilde{R}_{\nu}f(x) = \pv \int_0^{\infty} \widetilde{R}_{\nu}(x,y) f(y)\, d\eta_{\nu}(y), \qquad x > 0,
$$
valid for $f \in C_c^{\infty}(\mathbb{R}_+)$, where the kernel is
\begin{equation} \label{75}
\widetilde{R}_{\nu}(x,y) = (xy)^{-2\nu} \frac{1}{\sqrt{\pi}} \int_0^{\infty} \Big( \frac{\partial}{\partial x} W_t^{-\nu}(x,y)
	- \frac{2\nu}x W_t^{-\nu}(x,y) + \chi_{\{\nu \ge 1/2\}} \frac{2\nu}x W_t^{-\nu}(0,y) \Big) \, \frac{dt}{\sqrt{t}},
\end{equation}
being $W_t^{-\nu}(0,y) = \lim_{x\to 0^+} W_t^{-\nu}(x,y)$.

Using formulas \eqref{bes:dif} and \eqref{bes:rec} one computes that, for $\nu < 0$,
\begin{align*}
\frac{\partial}{\partial x} W_t^{-\nu}(x,y) - \frac{2\nu}x W_t^{-\nu}(x,y)
	& = \frac{(xy)^{\nu}}{(2t)^2} \exp\bigg( -\frac{x^2+y^2}{4t}\bigg)
		\bigg( y I_{-\nu-1}\Big(\frac{xy}{2t}\Big) - x I_{-\nu}\Big(\frac{xy}{2t}\Big)\bigg) \\
	& = -(xy)^{-1} \frac{\partial}{\partial x} W_t^{-\nu-1}(y,x).
\end{align*}
This means that
$$
\widetilde{R}_{\nu}(x,y) = - (xy)^{-2\nu-1} R_{-\nu-1}^{*}(x,y), \qquad \nu < 0,
$$
so the corresponding operators are related in the restricted range of $\nu$,
$$
\widetilde{R}_{\nu}f(x) = -x^{-2\nu-1} R_{-\nu-1}^{*}\big( y^{2\nu+1}f(y)\big)(x), \qquad \nu < 0.
$$

The last relation can be seen more directly on a formal level as follows. 
Recall that $B_{\nu} = D_{\nu}^{*} D$, $R_{\nu} = D(D_{\nu}^* D)^{-1/2}$, and $R_{\nu}^* = D_{\nu}^*(D D_{\nu}^*)^{-1/2}$.
Denote $V_{\nu}f(x) = x^{-2\nu-1}f(x)$, the inverse mapping being $V_{\nu}^{-1} = V_{-\nu-1}$.
Observe that $D_{\nu}^{*} = - V_{\nu} D V^{-1}_{\nu}$, so $D_{-\nu-1}^{*} = -V_{\nu}^{-1} D V_{\nu}$.
Further, $B_{\nu} = V_{\nu}(D D_{-\nu-1}^{*})V_{\nu}^{-1}$. Therefore,
$\widetilde{R}_{\nu} = -V_{\nu} R_{-\nu-1}^{*} V_{\nu}^{-1}$ and $\widetilde{R}_{\nu}^{*} = - V_{\nu} R_{-\nu-1} V_{\nu}^{-1}$
for $\nu < 0$.

Similarly as $R_{\nu}$,
the exotic Riesz transform $\widetilde{R}_{\nu}$ extends uniquely to a bounded operator on $L^2(d\eta_{\nu})$.
This extension is represented by means of the exotic (modified) Hankel transform (cf.\ \cite[Section 4]{NSS})
$$
\widetilde{h}_{\nu}f(z) = \int_0^{\infty} f(x) \widetilde{\varphi}_z^{\nu}(x)\, d\eta_{\nu}(x) =  z^{-2\nu} h_{-\nu}(x^{2\nu}f)(z),
\qquad \nu < 1,
$$
where $\widetilde{\varphi}_{z}^{\nu}(x) = (xz)^{-2\nu}\varphi_z^{-\nu}(x)$.
In fact, for $\nu < 1$ we have
\begin{align*}
\widetilde{R}_{\nu}f(x) & = - x^{-1}\widetilde{h}_{\nu-1}\big(y\widetilde{h}_{\nu}f(y)\big)(x), \qquad f \in L^2(d\eta_{\nu}),\\
\widetilde{R}_{\nu}^{*}f(x) & = - \widetilde{h}_{\nu}\big(y^{-1}\widetilde{h}_{\nu-1}\big(zf(z)\big)(y)\big)(x),
	\qquad f \in L^2(d\eta_{\nu}).
\end{align*}
To verify this one uses, among other things, the formula
$\frac{d}{dz}\widetilde{\varphi}_z^{\nu}(x) = -z^{-1}\widetilde{\varphi}_z^{\nu-1}(x)$, which in turn is a consequence of
the Bessel function differentiation rule \eqref{bes:dif}.
Note the $L^2(d\eta_{\nu})$ identities
$$
\widetilde{R}_{\nu}^{*} \widetilde{R}_{\nu} = \widetilde{R}_{\nu} \widetilde{R}_{\nu}^{*} = \textrm{Id}.
$$
Note also that $\widetilde{R}_{\nu}^{*}$ admits a singular integral representation on $C_c^{\infty}(\mathbb{R}_+)$,
with the kernel $\widetilde{R}_{\nu}^{*}(x,y) = \widetilde{R}_{\nu}(y,x)$.

In \cite[Theorem 2.3]{BHNV} and \cite[Proposition 2.4]{BHNV} the following characterization of mapping properties
of $R_{\nu}$ and $R_{\nu}^*$ was obtained.
\begin{thm}[{\cite{BHNV}}] \label{thm:Rcls}
Let $\nu > -1$, $1 \le p < \infty$, $\delta \in \mathbb{R}$. Then the operators $R_{\nu}$ and $R_{\nu}^{*}$, considered
on the measure space $(\mathbb{R}_+,x^{\delta}dx)$, have the following mapping properties:
\begin{itemize}
\item[(a1)] $R_{\nu}$ is of strong type $(p,p)$ if and only if $p > 1$ and $-1-p < \delta < (2\nu+2)p-1$;
\item[(a2)] $R_{\nu}^{*}$ is of strong type $(p,p)$ if and only if $p > 1$ and $-1 < \delta < 2(\nu+3/2)p -1$;
\item[(b1)] $R_{\nu}$ is of weak type $(p,p)$ if and only if $-1-p < \delta < (2\nu+2)p-1$,
	with both inequalities weakened in case $p=1$;
% or $\delta \in \{-2,2\nu+1\}$;
\item[(b2)] $R_{\nu}^{*}$ is of weak type $(p,p)$ if and only if $-1 < \delta < 2(\nu+3/2)p-1$,
	with the second inequality weakened in case $p=1$;
% or $\delta = 2\nu+2$;
\item[(c1)] $R_{\nu}$ is of restricted weak type $(p,p)$ if and only if $-1 - p \le \delta \le (2\nu+2)p-1$;
\item[(c2)] $R_{\nu}^{*}$ is of restricted weak type $(p,p)$ if and only if $-1 < \delta \le 2(\nu+3/2)p-1$.
\end{itemize}
Moreover,
$$
R_{\nu}^{*}R_{\nu}f = R_{\nu} R_{\nu}^{*}f = f, \qquad f \in L^p(\mathbb{R}_+,x^{\delta}dx),
$$
provided that $p > 1$ and $-1 < \delta < (2\nu+2)p-1$.
\end{thm}

The main result of this section is an analogous characterization in the exotic case.
It is convenient to write separate statements for $\nu < 1/2$ and $\nu \ge 1/2$, Theorems \ref{thm:Rexo}
and \ref{thm:Rexob}, respectively, since the two cases differ qualitatively.
Visualizations of these results are given on Figures \ref{fig_riesz_set1} and \ref{fig_riesz_set2} below.
\begin{thm} \label{thm:Rexo}
Let $0 \neq \nu < 1/2$, $1 \le p < \infty$, $\delta \in \mathbb{R}$. Then the exotic Riesz operators $\widetilde{R}_{\nu}$
and $\widetilde{R}_{\nu}^{*}$, considered on the measure space $(\mathbb{R}_+,x^{\delta}dx)$, have the following mapping
properties:
\begin{itemize}
\item[(a1)] $\widetilde{R}_{\nu}$ is of strong type $(p,p)$ if and only if $p > 1$ and $(2\nu+1)p-1 < \delta < 2p-1$;
\item[(a2)] $\widetilde{R}_{\nu}^{*}$ is of strong type $(p,p)$ if and only if $p > 1$ and $2\nu p -1 < \delta < p-1$;
\item[(b1)] $\widetilde{R}_{\nu}$ is of weak type $(p,p)$ if and only if $(2\nu+1)p-1 < \delta < 2p-1$, with the second inequality
	weakened in case $p=1$, and with the first inequality weakened in case $p=1$ and $\nu \neq -1/2$;
\item[(b2)] $\widetilde{R}_{\nu}^{*}$ is of weak type $(p,p)$ if and only if $2\nu p -1 < \delta < p-1$, with both inequalities
	weakened in case $p=1$;
\item[(c1)] $\widetilde{R}_{\nu}$ is of restricted weak type $(p,p)$ if and only if $(2\nu+1)p-1 \le \delta \le 2p-1$,
	with the first inequality strictened	 in case $\nu=-1/2$;
\item[(c2)] $\widetilde{R}_{\nu}^{*}$ is of restricted weak type $(p,p)$ if and only if $2\nu p -1 \le \delta \le p-1$.
\end{itemize}
Moreover,
$$
\widetilde{R}_{\nu}^{*}\widetilde{R}_{\nu}f = \widetilde{R}_{\nu} \widetilde{R}_{\nu}^{*}f = f,
	\qquad f \in L^p(\mathbb{R}_+,x^{\delta}dx),
$$
provided that $p > 1$ and $(2\nu+1)p -1 < \delta < p-1$.
\end{thm}

Observe that in Theorem \ref{thm:Rexo} the ranges of $\delta$ shrink as $\nu \to (1/2)^{-}$ and
become empty for $1/2 < \nu < 1$. However, this is not the case of the statement below.

\begin{thm} \label{thm:Rexob}
Let $1/2 \le \nu < 1$, $1 \le p < \infty$, $\delta \in \mathbb{R}$. Then the Riesz operators $\widetilde{R}_{\nu}$
and $\widetilde{R}_{\nu}^{*}$, considered on the measure space $(\mathbb{R}_+,x^{\delta}dx)$, have the following mapping
properties:
\begin{itemize}
\item[(a1)] $\widetilde{R}_{\nu}$ is of strong type $(p,p)$ if and only if $p > 1$ and $(2\nu-1)p-1 < \delta < 2p-1$;
\item[(a2)] $\widetilde{R}_{\nu}^{*}$ is of strong type $(p,p)$ if and only if $p > 1$ and $2\nu p -1 < \delta < 3p-1$;
\item[(b1)] $\widetilde{R}_{\nu}$ is of weak type $(p,p)$ if and only if $(2\nu-1)p-1 < \delta < 2p-1$, with both inequalities
	weakened in case $p=1$ and $\nu \neq 1/2$;
\item[(b2)] $\widetilde{R}_{\nu}^{*}$ is of weak type $(p,p)$ if and only if $2\nu p -1 < \delta < 3p-1$, with the 
	second inequality weakened in case $p=1$, and with the first inequality
	weakened in case $p=1$ and $\nu \neq 1/2$;
\item[(c1)] $\widetilde{R}_{\nu}$ is of restricted weak type $(p,p)$ if and only if $(2\nu-1)p-1 \le \delta \le 2p-1$,
	with both inequalities strictened	 in case $\nu=1/2$;
\item[(c2)] $\widetilde{R}_{\nu}^{*}$ is of restricted weak type $(p,p)$ if and only if $2\nu p -1 \le \delta \le 3p-1$,
	with the first inequality strictened in case $\nu =1/2$.
\end{itemize}
Moreover,
$$
\widetilde{R}_{\nu}^{*}\widetilde{R}_{\nu}f = \widetilde{R}_{\nu} \widetilde{R}_{\nu}^{*}f = f,
	\qquad f \in L^p(\mathbb{R}_+,x^{\delta}dx),
$$
provided that $p > 1$ and $2\nu p -1 < \delta <2p-1$.
\end{thm}

\begin{figure}
\centering

\begin{tikzpicture}[scale=1.5]

\draw[arrows=-angle 60] (0,0) -- (0,2.5);
\draw[arrows=-angle 60] (0,0) -- (8,0);

\node at (0.2,2.3) {$p$};
\node at (7.8,0.2) {$\delta$};

\draw[very thin] (2.1,-0.05) -- (2.1,0.05);
\node at (2.1, -0.2) {$-2$};
\draw[very thin] (2.8,-0.05) -- (2.8,0.05);
\node at (2.8, -0.2) {$-1$};
\draw[very thin] (4.2,-0.05) -- (4.2,0.05);
\node at (4.3, -0.2) {$2\nu+2$};
\draw[very thin] (3.5,-0.05) -- (3.5,0.05);
\node at (3.5,-0.2) {$2\nu+1$};

\draw[very thin, color=black] (-0.05,0.7) -- (0.05,0.7);
\node at (-0.2,0.7) {$1$};

\draw[very thin] (-0.05,1.4) -- (0.05,1.4);
\node at (-0.2,1.4) {$2$};

\fill[black!8!white] (1.1,1.7) -- (2.1,0.68) -- (3.5,0.68) -- (4.5,1.7);
\fill[pattern=my dots]  (2.8,1.7) -- (2.8,0.7) -- (4.2,0.7) -- (6.2,1.7);

\draw[thick, dash dot] (4.2, 0.7) -- (6.2,1.7);

\draw[thick, dash dot] (2.1,0.68) -- (1.1,1.7);
\draw[thick, dash dot] (3.5,0.68) -- (4.5,1.7);

\draw[thin, dashed] (2.8, 0.7) -- (2.8,1.7);
\draw[thin] (2.8,0.7) -- (4.2,0.7);
\draw[thin] (2.1,0.68) -- (3.5,0.68);

\filldraw[fill=black] (4.2,0.7) circle(1pt);
\filldraw[fill=white] (2.8,0.7) circle(1pt);

\filldraw[fill=black] (2.1,0.68) circle(1pt);
\filldraw[fill=black] (3.5,0.68) circle(1pt);

\node[rotate=-45] at (1.46, 1.06) {\scriptsize r.w.t. $(p,p)$};

\node[rotate=45] at (4.21, 1.18) {\scriptsize r.w.t. $(p,p)$};

\node[rotate=26.57] at (5.28, 1.07) {\scriptsize r.w.t. $(p,p)$};

\node at (7,2.2) {\textbf{CLS, $\nu>-1$}};
 \end{tikzpicture} 

\centering

\begin{tikzpicture}[scale=1.5]

\draw[arrows=-angle 60] (0,0) -- (0,2.5);
\draw[arrows=-angle 60] (0,0) -- (8,0);

\node at (0.2,2.3) {$p$};
\node at (7.8,0.2) {$\delta$};

\draw[very thin] (1.75,-0.05) -- (1.75,0.05);
\node at (1.75, -0.2) {$2\nu-1$};
\draw[very thin] (2.45,-0.05) -- (2.45,0.05);
\node at (2.45,-0.2) {$2\nu$};
\draw[very thin] (3.5,-0.05) -- (3.5,0.05);
\node at (3.5,-0.2) {$0$};
\draw[very thin] (4.2,-0.05) -- (4.2,0.05);
\node at (4.2,-0.2) {$1$};

\draw[very thin, color=black] (-0.05,0.7) -- (0.05,0.7);
\node at (-0.2,0.7) {$1$};

\draw[very thin] (-0.05,1.4) -- (0.05,1.4);
\node at (-0.2,1.4) {$2$};

\fill[black!8!white] (1.95,1.7) -- (2.45,0.68) -- (4.2,0.68) -- (6.3,1.7);
\fill[pattern=my dots]  (0.25,1.7) -- (1.75,0.7) -- (3.5,0.7) -- (4.5,1.7);

\draw[thick, dash dot] (1.75, 0.7) -- (0.25,1.7);
\draw[thick, dash dot] (3.5,0.7) -- (4.5, 1.7);
\draw[thick, dash dot] (2.45,0.68) -- (1.95,1.7);
\draw[thick, dash dot] (4.2,0.68) -- (6.3,1.7);

\draw[thin] (1.75,0.7) -- (3.5,0.7);
\draw[thin] (2.45,0.68) -- (4.2,0.68);

\filldraw[fill=black] (1.75,0.7) circle(1pt);
\filldraw[fill=black] (3.5,0.7) circle(1pt);

\filldraw[fill=black] (2.45,0.68) circle(1pt);
\filldraw[fill=black] (4.2,0.68) circle(1pt);

\node[rotate=-63.43] at (2.01, 1.17) {\scriptsize r.w.t. $(p,p)$};
\node[rotate=-33.69] at (0.894452,1.0311679) {\scriptsize r.w.t. $(p,p)$};
\node at (3.325,0.5) {\scriptsize weak $(1,1)$};
\node[rotate=45] at (4.17, 1.13) {\scriptsize r.w.t. $(p,p)$};
\node[rotate=26.565] at (5.3, 1.01) {\scriptsize r.w.t. $(p,p)$};
\node at (7,2.2) {\textbf{EXO, $\nu<-1/2$}};
 \end{tikzpicture}

\begin{tikzpicture}[scale=1.5]

\draw[arrows=-angle 60] (0,0) -- (0,2.5);
\draw[arrows=-angle 60] (0,0) -- (8,0);

\node at (0.2,2.3) {$p$};
\node at (7.8,0.2) {$\delta$};

\draw[very thin] (2.1,-0.05) -- (2.1,0.05);
\node at (2.1,-0.2) {$2\nu-1$};

\draw[very thin] (2.8,-0.05) -- (2.8,0.05);
\node at (2.8,-0.2) {$2\nu$};

\draw[very thin] (3.5,-0.05) -- (3.5,0.05);
\node at (3.5,-0.2) {$0$};

\draw[very thin] (4.2,-0.05) -- (4.2,0.05);
\node at (4.2,-0.2) {$1$};

\draw[very thin, color=black] (-0.05,0.7) -- (0.05,0.7);
\node at (-0.2,0.7) {$1$};

\draw[very thin] (-0.05,1.4) -- (0.05,1.4);
\node at (-0.2,1.4) {$2$};

\fill[black!8!white]  (2.8,1.7) -- (2.8,0.7) -- (4.2,0.7) -- (6.3,1.7);
 \fill[pattern=my dots](1.1,1.7) -- (2.1,0.7) -- (3.5,0.7) -- (4.5,1.7); 

\draw[thick, dash dot] (2.1,0.7) -- (1.1,1.7);
\draw[thick, dash dot] (3.5,0.7) -- (4.5,1.7);
\draw[thick, dash dot] (4.2,0.68) -- (6.3,1.7);

\draw[ thin] (2.1,0.7) -- (3.5,0.7);
\draw[thin] (2.8,0.68) -- (4.2,0.68);

\draw[thin, dashed] (2.8,0.7) -- (2.8,1.7);

\filldraw[fill=black] (2.1,0.7) circle(1pt);
\filldraw[fill=white] (2.8,0.68) circle(1pt);
\filldraw[fill=black] (3.5,0.7) circle(1pt);
\filldraw[fill=black] (4.2,0.68) circle(1pt);

\node[rotate=-45] at (1.43, 1.13) {\scriptsize r.w.t. $(p,p)$};
\node at (3.15,0.5) {\scriptsize weak $(1,1)$};
\node[rotate=45] at (4.17, 1.13) {\scriptsize r.w.t. $(p,p)$};
\node[rotate=26.565] at (5.3, 1.01) {\scriptsize r.w.t. $(p,p)$};
\node at (7,2.2) {\textbf{EXO, $\nu=-1/2$}};
 \end{tikzpicture} 

\hspace*{1em}
\begin{tikzpicture}[scale=1.5]

\draw[arrows=-angle 60] (0,0) -- (0,2.5);
\draw[arrows=-angle 60] (0,0) -- (8,0);

\node at (0.2,2.3) {$p$};
\node at (7.8,0.2) {$\delta$};

\draw[very thin] (2.45,-0.05) -- (2.45,0.05);
\node at (2.45, -0.2) {$2\nu-1$};
\draw[very thin] (3.15,-0.05) -- (3.15,0.05);
\node at (3.15,-0.2) {$2\nu$};
\draw[very thin] (3.5,-0.05) -- (3.5,0.05);
\node at (3.5,-0.2) {$0$};
\draw[very thin] (4.2,-0.05) -- (4.2,0.05);
\node at (4.2,-0.2) {$1$};

\draw[very thin, color=black] (-0.05,0.7) -- (0.05,0.7);
\node at (-0.2,0.7) {$1$};

\draw[very thin] (-0.05,1.4) -- (0.05,1.4);
\node at (-0.2,1.4) {$2$};

 \fill[black!8!white](3.65,1.7) -- (3.15,0.68) -- (4.2,0.68) -- (6.3,1.7);
 \fill[pattern=my dots] (1.95,1.7) -- (2.45,0.7) -- (3.5,0.7) -- (4.5,1.7);

\draw[thick, dash dot] (2.45, 0.7) -- (1.95,1.7);
\draw[thick, dash dot] (3.15,0.68) -- (3.65, 1.7);
\draw[thick, dash dot] (3.5,0.7) -- (4.5,1.7);
\draw[thick, dash dot] (4.2,0.68) -- (6.3,1.7);

\draw[thin] (2.45,0.7) -- (3.5,0.7);
\draw[thin] (3.15,0.68) -- (4.2,0.68);

\filldraw[fill=black] (2.45,0.7) circle(1pt);
\filldraw[fill=black] (3.5,0.7) circle(1pt);
\filldraw[fill=black] (3.15,0.68) circle(1pt);
\filldraw[fill=black] (4.2,0.68) circle(1pt);

\node[rotate=-63.43] at (2.02, 1.15) {\scriptsize r.w.t. $(p,p)$};
\node[rotate=63.43] at (3.23, 1.24) {\scriptsize r.w.t. $(p,p)$};
\node at (3.325,0.5) {\scriptsize weak $(1,1)$};
\node[rotate=45] at (4.17, 1.13) {\scriptsize r.w.t. $(p,p)$};
\node[rotate=26.565] at (5.3, 1.01) {\scriptsize r.w.t. $(p,p)$};
\node at (7,2.2) {\textbf{EXO, $-1/2<\nu<0$}};
 \end{tikzpicture} 

\hspace*{1em}
\begin{tikzpicture}[scale=1.5]

\draw[arrows=-angle 60] (0,0) -- (0,2.5);
\draw[arrows=-angle 60] (0,0) -- (8,0);

\node at (0.2,2.3) {$p$};
\node at (7.8,0.2) {$\delta$};

\draw[very thin] (3.15,-0.05) -- (3.15,0.05);
\node at (3, -0.2) {$2\nu-1$};

\draw[very thin] (3.5,-0.05) -- (3.5,0.05);
\node at (3.5,0.2) {$0$};

\draw[very thin] (3.85,-0.05) -- (3.85,0.05);
\node at (3.85,-0.2) {$2\nu$};

\draw[very thin] (4.2,-0.05) -- (4.2,0.05);
\node at (4.2,-0.2) {$1$};

\draw[very thin, color=black] (-0.05,0.7) -- (0.05,0.7);
\node at (-0.2,0.7) {$1$};

\draw[very thin] (-0.05,1.4) -- (0.05,1.4);
\node at (-0.2,1.4) {$2$};

\fill[black!8!white] (5.35,1.7) -- (3.85,0.7) -- (4.2,0.7) -- (6.3,1.7);
\fill[pattern=my dots](3.65,1.7) -- (3.15,0.7) -- (3.5,0.7) -- (4.5,1.7); 

\draw[thick, dash dot] (3.15,0.7) -- (3.65,1.7);
\draw[thick, dash dot] (3.5,0.7) -- (4.5,1.7);
\draw[thick, dash dot] (3.85,0.7) -- (5.35, 1.7);
\draw[thick, dash dot] (4.2,0.7) -- (6.3,1.7);

\draw[thin] (3.15,0.7) -- (3.5,0.7);
\draw[thin] (3.85,0.7) -- (4.2,0.7);

\filldraw[fill=black] (3.15,0.7) circle(1pt);
\filldraw[fill=black] (3.5,0.7) circle(1pt);
\filldraw[fill=black] (3.85,0.7) circle(1pt);
\filldraw[fill=black] (4.2,0.7) circle(1pt);

\node[rotate=63.43] at (3.23, 1.24) {\scriptsize r.w.t. $(p,p)$};
\node at (3.675,0.5) {\scriptsize weak $(1,1)$};

\node[rotate=42] at (4.25, 1.2) {\scriptsize r.w.t. $(p,p)$};

\node[rotate=26.565] at (5.3, 1.01) {\scriptsize r.w.t. $(p,p)$};
\node at (7,2.2) {\textbf{EXO, $0<\nu<1/2$}};
 \end{tikzpicture}

\caption{Visualization of Theorems \ref{thm:Rcls} and \ref{thm:Rexo}, with the choices
$\nu =-1/2$, $\nu=-3/4$, $\nu=-1/2$, $\nu=-1/4$ and $\nu=1/4$, respectively.
The gray regions correspond to $\widetilde{R}_{\nu}$, while the dotted to $\widetilde{R}_{\nu}^*$.
Note the phase shifts occurring at $\nu = -1/2$ and $\nu = 0$ in the exotic case.
}

\label{fig_riesz_set1}

\end{figure}
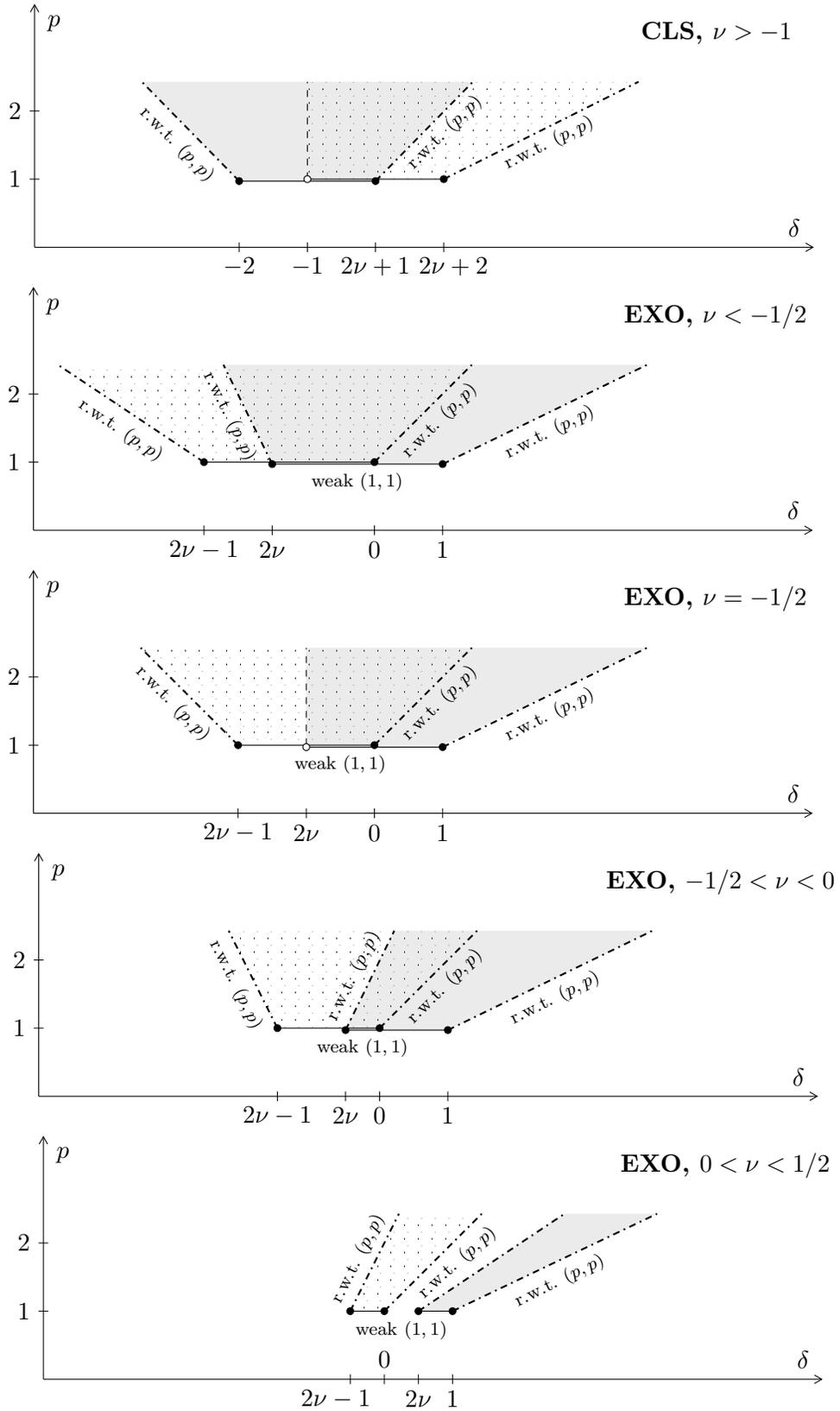

\begin{figure}
\centering 

\begin{tikzpicture}[scale=1.5]

\draw[arrows=-angle 60] (0,0) -- (0,2.5);
\draw[arrows=-angle 60] (0,0) -- (8,0);

\node at (0.2,2.3) {$p$};
\node at (7.8,0.2) {$\delta$};

\draw[very thin] (2.1,-0.05) -- (2.1,0.05);
\node at (2.1, -0.2) {$-2$};
\draw[very thin] (2.8,-0.05) -- (2.8,0.05);
\node at (2.8, -0.2) {$-1$};
\draw[very thin] (4.2,-0.05) -- (4.2,0.05);
\node at (4.3, -0.2) {$2\nu+2$};
\draw[very thin] (3.5,-0.05) -- (3.5,0.05);
\node at (3.5,-0.2) {$2\nu+1$};

\draw[very thin, color=black] (-0.05,0.7) -- (0.05,0.7);
\node at (-0.2,0.7) {$1$};

\draw[very thin] (-0.05,1.4) -- (0.05,1.4);
\node at (-0.2,1.4) {$2$};

\fill[black!8!white]  (1.1,1.7) -- (2.1,0.68) -- (3.5,0.68) -- (4.5,1.7);
 \fill[pattern=my dots](2.8,1.7) -- (2.8,0.7) -- (4.2,0.7) -- (6.2,1.7);

\draw[thick, dash dot] (4.2, 0.7) -- (6.2,1.7);

\draw[thick, dash dot] (2.1,0.68) -- (1.1,1.7);
\draw[thick, dash dot] (3.5,0.68) -- (4.5,1.7);

\draw[thin, dashed] (2.8, 0.7) -- (2.8,1.7);

\draw[thin] (2.8,0.7) -- (4.2,0.7);
\draw[thin] (2.1,0.68) -- (3.5,0.68);

\filldraw[fill=black] (4.2,0.7) circle(1pt);
\filldraw[fill=white] (2.8,0.7) circle(1pt);

\filldraw[fill=black] (2.1,0.68) circle(1pt);
\filldraw[fill=black] (3.5,0.68) circle(1pt);

\node[rotate=-45] at (1.46, 1.06) {\scriptsize r.w.t. $(p,p)$};

\node[rotate=45] at (4.21, 1.18) {\scriptsize r.w.t. $(p,p)$};

\node[rotate=26.57] at (5.28, 1.07) {\scriptsize r.w.t. $(p,p)$};

\node at (7,2.2) {\textbf{CLS, $\nu>-1$}};
 \end{tikzpicture} 

\begin{tikzpicture}[scale=1.5]

\draw[arrows=-angle 60] (0,0) -- (0,2.5);
\draw[arrows=-angle 60] (0,0) -- (8,0);

\node at (0.2,2.3) {$p$};
\node at (7.8,0.2) {$\delta$};

\draw[very thin] (2.8,-0.05) -- (2.8,0.05);
\node at (2.7,-0.2) {$2\nu-2$};

\draw[very thin] (3.5,-0.05) -- (3.5,0.05);
\node at (3.6,-0.2) {$2\nu-1$};

\draw[very thin] (4.2,-0.05) -- (4.2,0.05);
\node at (4.2,-0.2) {$1$};

\draw[very thin] (4.9,-0.05) -- (4.9,0.05);
\node at (4.9,-0.2) {$2$};

\draw[very thin, color=black] (-0.05,0.7) -- (0.05,0.7);
\node at (-0.2,0.7) {$1$};

\draw[very thin] (-0.05,1.4) -- (0.05,1.4);
\node at (-0.2,1.4) {$2$};

\fill[black!8!white]  (2.8,1.7) -- (2.8,0.68) -- (4.2,0.68) -- (6.2,1.7);
\fill[pattern=my dots] (4.5,1.7) -- (3.5,0.7) -- (4.9,0.7) -- (7.9,1.7); 

\draw[thin, dashed] (2.8,0.68) -- (2.8,1.7);
\draw[thin, dashed] (4.2,0.68) -- (6.2,1.7);

\draw[thin, dashed] (3.5,0.7) -- (4.5, 1.7);
\draw[thick, dash dot] (4.9,0.7) -- (7.9,1.7);

\draw[thin] (3.5,0.7) -- (4.9,0.7);
\draw[thin] (2.8,0.68) -- (4.2,0.68);

\filldraw[fill=white] (3.5,0.7) circle(1pt);
\filldraw[fill=black] (4.9,0.7) circle(1pt);
\filldraw[fill=white] (2.8,0.68) circle(1pt);
\filldraw[fill=white] (4.2,0.68) circle(1pt);

\node at (3.85,0.5) {\scriptsize weak $(1,1)$};
\node[rotate=18.43] at (6.45, 1.06) {\scriptsize r.w.t. $(p,p)$};
\node at (7,2.2) {\textbf{EXO, $\nu=1/2$}};
 \end{tikzpicture} 

\begin{tikzpicture}[scale=1.5]

\draw[arrows=-angle 60] (0,0) -- (0,2.5);
\draw[arrows=-angle 60] (0,0) -- (8,0);

\node at (0.2,2.3) {$p$};
\node at (7.8,0.2) {$\delta$};

\draw[very thin] (3.15,-0.05) -- (3.15,0.05);
\node at (3.15,-0.2) {$2\nu-2$};

\draw[very thin] (4.2,-0.05) -- (4.2,0.05);
\node at (4.2,-0.2) {$1$};

\draw[very thin] (3.85,-0.05) -- (3.85,0.05);
\node at (3.85,0.2) {$2\nu-1$};

\draw[very thin] (4.9,-0.05) -- (4.9,0.05);
\node at (4.9,-0.2) {$2$};

\draw[very thin, color=black] (-0.05,0.7) -- (0.05,0.7);
\node at (-0.2,0.7) {$1$};

\draw[very thin] (-0.05,1.4) -- (0.05,1.4);
\node at (-0.2,1.4) {$2$};

\fill[black!8!white] (3.65,1.7) -- (3.15,0.68) -- (4.2,0.68) -- (6.2,1.7);
\fill[pattern=my dots] (5.35,1.7) -- (3.85,0.7) -- (4.9,0.7) -- (7.9,1.7); 

\draw[thick, dash dot] (3.15,0.68) -- (3.65,1.7);
\draw[thick, dash dot] (4.2,0.68) -- (6.2,1.7);

\draw[thick, dash dot] (3.85,0.7) -- (5.35, 1.7);
\draw[thick, dash dot] (4.9,0.7) -- (7.9,1.7);

\draw[thin] (3.85,0.7) -- (4.9,0.7);
\draw[thin] (3.15,0.68) -- (4.2,0.68);

\filldraw[fill=black] (3.85,0.7) circle(1pt);
\filldraw[fill=black] (4.9,0.7) circle(1pt);

\filldraw[fill=black] (3.15,0.68) circle(1pt);
\filldraw[fill=black] (4.2,0.68) circle(1pt);

\node[rotate=63.43] at (3.22,1.28) {\scriptsize r.w.t. $(p,p)$};

\node[rotate=26.57] at (5.27, 1.06) {\scriptsize r.w.t. $(p,p)$};

\node[rotate=18.43] at (6.45, 1.06) {\scriptsize r.w.t. $(p,p)$};

\node[rotate=37] at (4.32,1.23) {\scriptsize r.w.t. $(p,p)$};

\node at (4.025,0.5) {\scriptsize weak $(1,1)$};
\node at (7,2.2) {\textbf{EXO, $\nu>1/2$}};
\end{tikzpicture} 

\caption{Visualization of Theorems \ref{thm:Rcls} and  \ref{thm:Rexob}, with the choices
$\nu =-1/2$, $\nu=1/2$ and $\nu=3/4$, respectively.
The gray regions correspond to $\widetilde{R}_{\nu}$, while the dotted to $\widetilde{R}_{\nu}^*$.
Note the phase shift occurring at $\nu = 1/2$ in the exotic case.}

\label{fig_riesz_set2}

\end{figure}
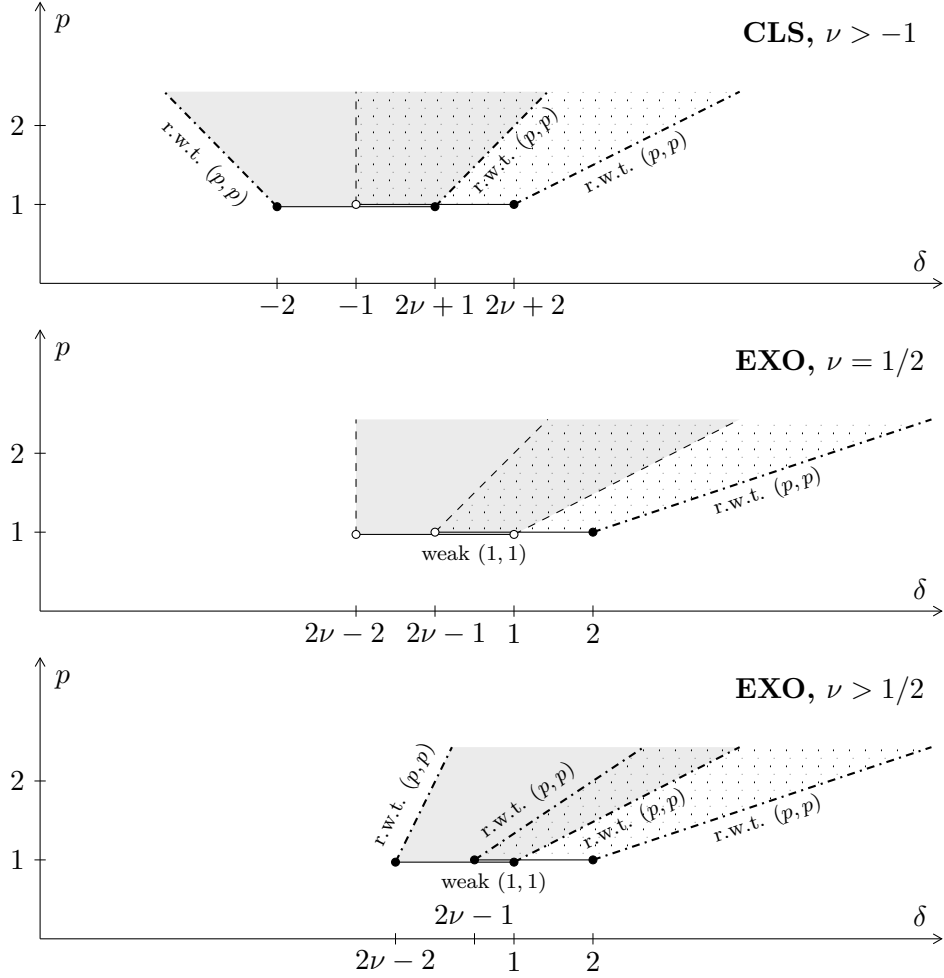

The ranges of $\delta$ in Theorem \ref{thm:Rexob} are wider than in Theorem \ref{thm:Rexo}. This is thanks to the compensation
in the definition of the potential. On the other hand, we would not be able to define the Riesz operators without that
compensation for some parameters. Note that introducing the compensation also for $\nu < 1/2$ would give wider ranges in
Theorem \ref{thm:Rexo}, actually the same as in Theorem \ref{thm:Rexob}.
This phenomenon is absent in the classical situation, where extending the compensation in the potential operator
to $\nu > 1/2$ would not affect the Riesz transforms.

Let $\delta = 2\nu+1$, so that $x^{\delta}dx = d\eta_{\nu}$ is the natural weight.
Then, in view of Theorems \ref{thm:Rexo} and \ref{thm:Rexob}, there are always some $p$ for which
$\widetilde{R}_{\nu}$ is bounded on $L^p(d\eta_{\nu})$, and the same is true about $\widetilde{R}_{\nu}^*$.
Moreover, when $\nu \le -1/2$ the behavior of $\widetilde{R}_{\nu}$ and $\widetilde{R}_{\nu}^*$ is standard,
both the operators are bounded on $L^p(d\eta_{\nu})$, $1< p < \infty$, and from $L^1(d\eta_{\nu})$ to weak
$L^1(d\eta_{\nu})$. On the other hand, when $-1/2 < \nu < 1$, $\nu \neq 0$, the ranges of $p$ for
$L^p(d\eta_{\nu})$-boundedness are restricted. More precisely, $\widetilde{R}_{\nu}$ is bounded on $L^p(d\eta_{\nu})$
if and only if $\nu+1 < p < 1 + \frac{1}{2\nu+1}$ in case $-1/2 < \nu < 1/2$ or $\nu+1 < p < 1+\frac{3}{2\nu-1}$
in case $1/2 \le \nu < 1$. Here the lower bound $\nu + 1 < p$ is meaningful only for $\nu > 0$, and
$\frac{3}{2\nu-1}$ is understood as $\infty$ for $\nu = 1/2$.
Similarly, $\widetilde{R}_{\nu}^*$ is bounded on $L^p(d\eta_{\nu})$ if and only if $2\nu+2 < p < 1+\frac{1}{\max(\nu,0)}$
in case $-1/2 < \nu < 1/2$ or $\frac{2\nu+2}{3} < p < 1 + \frac{1}{\nu}$ in case $1/2 \le \nu < 1$. Here
the upper bound $p < 1+\frac{1}{\max(\nu,0)}$ is understood to be $\infty$ for $\nu < 0$, so it is meaningful only for
$\nu > 0$.

To prove Theorems \ref{thm:Rexo} and \ref{thm:Rexob} we shall need several auxiliary results gathered below.

\begin{pro} \label{prop:kere12}
The kernel $\widetilde{R}_{\nu}(x,y)$ for $\nu=1/2$ has the explicit form
$$
\widetilde{R}_{1/2}(x,y) = \frac{1}{\pi xy}\bigg( \frac{1}{y-x} - \frac{1}{y+x} - \frac{1}{x} \log\frac{y^2}{|x-y|(x+y)}\bigg),
	\qquad x \neq y.
$$
\end{pro}

\begin{proof}
First, we compute the integral kernel of the classical compensated potential $(B_{-1/2}^{\textrm{cls}})^{-1/2}$,
call it $K_{-1/2}(x,y)$. We have
$$
K_{-1/2}(x,y) = \frac{1}{\sqrt{\pi}} \int_0^{\infty} \Big( W_t^{-1/2}(x,y) - W_t^{-1/2}(0,y) \Big) \, \frac{dt}{\sqrt{t}}.
$$ 
Using \eqref{bes:el} we see that
$$
W_{t}^{-1/2}(x,y) = \frac{1}{\sqrt{4 \pi t}}\bigg( e^{-\frac{(x-y)^2}{4t}} + e^{-\frac{(x+y)^2}{4t}}\bigg)
$$
and, consequently,
$$
K_{-1/2}(x,y) = \frac{1}{2\pi} \int_0^{\infty} \bigg( e^{-\frac{(x-y)^2}{4t}} + e^{-\frac{(x+y)^2}{4t}}
	- 2 e^{-\frac{y^2}{4t}}\bigg)\, \frac{dt}t.
$$
The last integral can be written as a limit $\lim_{N\to \infty} \int_0^N \ldots$, and then simple changes of variables
lead to
\begin{align*}
K_{-1/2}(x,y) & = \lim_{N\to \infty}\frac{1}{2\pi} \bigg\{ \int_{(x-y)^2/4N}^{\infty} + \int_{(x+y)^2/4N}^{\infty}
		- 2\int_{y^2/4N}^{\infty} \bigg\} \, e^{-t}\, \frac{dt}{t} \\
	& = \lim_{N\to \infty}\frac{1}{2\pi} \bigg\{ \int_{(x-y)^2/4N}^{y^2/4N} - \int_{y^2/4N}^{(x+y)^2/4N}\bigg\}
	\, e^{-t}\, \frac{dt}{t}.
\end{align*}
On the other hand, given $A>0$ fixed,
$$
\lim_{s\to 0^+} \int_s^{As} e^{-t} \, \frac{dt}{t} = \log A + \lim_{s\to 0^+} \int_s^{As} \frac{e^{-t}-1}{t}\, dt = \log A.
$$
Combining the above formulas one gets
$$
K_{-1/2}(x,y) = \frac{1}{2\pi} \bigg( \log\frac{y^2}{(x-y)^2} - \log\frac{(x+y)^2}{y^2} \bigg)
	= \frac{1}{\pi} \log\frac{y^2}{|x-y|(x+y)}.
$$

Next, we observe that, see \eqref{75},
$$
\widetilde{R}_{1/2}(x,y) = (xy)^{-1} \Big( R_{-1/2}(x,y) - \frac{1}x K_{-1/2}(x,y)\Big).
$$
This, the formula for $K_{-1/2}(x,y)$ above and the formula (see \cite[(4.3)]{BHNV})
$$
R_{-1/2}(x,y) = \frac{1}{\pi} \bigg( \frac{1}{y-x} - \frac{1}{y+x} \bigg)
$$
give the desired conclusion.
\end{proof}

\begin{pro} \label{prop:keru12}
The kernel $\widetilde{R}_{\nu}(x,y)$ for $\nu = 1/2$ satisfies, for $x,y > 0$, $x \neq y$,
$$
\widetilde{R}_{1/2}(x,y) = \frac{1}{\pi} \frac{(xy)^{-1}}{y-x} + \mathcal{O}\bigg( \frac{1}{y^3}
	\Big( 1 + \log\frac{xy}{(y-x)^2}\Big) \bigg), \qquad x/2 < y < 2x.
$$
Moreover, in the off-diagonal region,
$$
\big|\widetilde{R}_{1/2}(x,y)\big| \lesssim
	\begin{cases}
		x^{-2} y^{-1} \log\frac{x}{y}, & y \le x/2, \\
		y^{-3}, & 2x \le y.
	\end{cases}
$$
\end{pro}

\begin{proof}
Simple analysis based on the explicit formula, see Proposition \ref{prop:kere12}.
\end{proof}

\begin{lem} \label{lem:Rlowe}
Let $0 \neq \nu < 1$ be fixed. There exist constants $b>1$ and $c>0$ such that
\begin{itemize}
\item[(a)]
in case $\nu < 0$
\begin{align*}
\widetilde{R}_{\nu}(x,y) & \le -c x^{-2} y^{-2\nu}, \qquad 0 < y \le x/b, \\
\widetilde{R}_{\nu}(x,y) & \ge c x^{-2\nu-1}y^{-1}, \qquad 0 < bx \le y;
\end{align*}
\item[(b)]
in case $0 < \nu < 1/2$
\begin{align*}
\widetilde{R}_{\nu}(x,y) & \le -c x^{-2} y^{-2\nu}, \qquad 0 < y \le x/b, \\
\widetilde{R}_{\nu}(x,y) & \le - c x^{-2\nu-1}y^{-1}, \qquad 0 < bx \le y;
\end{align*}
\item[(c)]
in case $\nu = 1/2$
\begin{align*}
\widetilde{R}_{\nu}(x,y) & \ge c x^{-2} y^{-1} \log\frac{x}y, \qquad 0 < y \le x/b, \\
\widetilde{R}_{\nu}(x,y) & \ge c y^{-3}, \qquad 0 < bx \le y.
\end{align*}
\item[(d)]
in case $1/2 < \nu < 1$
\begin{align*}
\widetilde{R}_{\nu}(x,y) & \ge c x^{-2} y^{-2\nu}, \qquad 0 < y \le x/b, \\
\widetilde{R}_{\nu}(x,y) & \ge c x^{-2\nu+1}y^{-3}, \qquad 0 < bx \le y.
\end{align*}
\end{itemize}
\end{lem}

\begin{proof}
Recall that, see \eqref{75},
$$
\widetilde{R}_{\nu}(x,y) = (xy)^{-2\nu} R_{-\nu}(x,y) - (xy)^{-2\nu}\frac{2\nu}x K_{-\nu}(x,y),
$$
where the classical compensated potential kernel is given by
$$
K_{\mu}(x,y) = \frac{1}{\sqrt{\pi}} \int_0^{\infty} \Big( W_t^{\mu}(x,y) - \chi_{\{\mu \le -1/2\}} W_t^{\mu}(0,y)\Big)\,
	\frac{dt}{\sqrt{t}}, \qquad \mu > -1.
$$
From the proof of \cite[Lemma 4.4]{BHNV} we know that$^{\dag}$ for $\mu > -1$
\footnote{$\dag$ In the computations of the proof of \cite[Lemma 4.4]{BHNV} the factor $\pi^{-1/2}$
involved in the explicit expression for $R_{\lambda}(x,y)$ there is missing.}
\begin{align*}
\lim_{z \to 0^+} x^{2\mu+2} R_{\mu}(x,y) & = -\frac{2\Gamma(\mu+3/2)}{\sqrt{\pi}\Gamma(\mu+1)}, \qquad z = \frac{y}{x},\\
\lim_{z \to 0^+} x^{-1} y^{2\mu+3} R_{\mu}(x,y) & = \frac{\Gamma(\mu+3/2)}{\sqrt{\pi} \Gamma(\mu+2)}, \qquad z = \frac{x}{y}.
\end{align*}

We consider first $0 \neq \nu < 1/2$, the case when no compensation in the potential occurs.
Here, we will need the following identities valid for $\mu > -1/2$,
\begin{align} \label{81}
\lim_{z \to 0^+} x^{2\mu+1} K_{\mu}(x,y) & = \frac{\Gamma(\mu+1/2)}{\sqrt{\pi}\Gamma(\mu+1)}, \qquad z = \frac{y}{x},\\
\lim_{z \to 0^+} y^{2\mu+1} K_{\mu}(x,y) & = \frac{\Gamma(\mu+1/2)}{\sqrt{\pi}\Gamma(\mu+1)}, \qquad z = \frac{x}{y}. \nonumber
\end{align}
These identities are equivalent, for symmetry reasons. So it is enough we verify the first one.
Using the formula for $K_{\mu}(x,y)$ and then changing the variable $u = 2t/x^2$ we get
\begin{align*}
K_{\mu}(x,y) & = \frac{1}{\sqrt{\pi}} \int_0^{\infty} W_t^{\mu}(x,y) \, \frac{dt}{\sqrt{t}} \\
	& = \frac{1}{\sqrt{\pi}} \int_0^{\infty} \frac{1}{2t} (xy)^{-\mu} e^{-\frac{x^2+y^2}{4t}}I_{\mu}\Big(\frac{xy}{2t}\Big)\,
		\frac{dt}{\sqrt{t}} \\
	& = \frac{1}{\sqrt{2\pi}} x^{-2\mu-1} \int_0^{\infty} \Big(\frac{z}{u}\Big)^{-\mu} I_{\mu}\Big(\frac{z}{u}\Big)
		e^{-\frac{1+z^2}{2u}} \, \frac{du}{u^{\mu+3/2}},
\end{align*}
with $z=y/x$. Denoting the last integral by $\mathcal{J}=\mathcal{J}(z)$ and using the dominated convergence theorem (see the proof of
\cite[Lemma 4.4]{BHNV}) we get
$$
\lim_{z \to 0^+} \mathcal{J}(z) = \frac{1}{2^{\mu}\Gamma(\mu+1)} \int_0^{\infty} e^{-\frac{1}{2u}}\, \frac{du}{u^{\mu+3/2}}
	= \frac{\sqrt{2}\Gamma(\mu+1/2)}{\Gamma(\mu+1)},
$$
where we applied \eqref{bes:lim}. This gives \eqref{81}.

Let now $z = y/x$. In view of what was said above,
\begin{align*}
\lim_{z \to 0^+} x^2 y^{2\nu} \widetilde{R}_{\nu}(x,y) & =
	\lim_{z \to 0^+} x^{-2\nu+2} R_{-\nu}(x,y) - 2\nu \lim_{z\to 0^+} x^{-2\nu+1} K_{-\nu}(x,y) \\
	& = -\frac{2\Gamma(-\nu+3/2)}{\sqrt{\pi}\Gamma(-\nu+1)} - 2\nu \frac{\Gamma(-\nu+1/2)}{\sqrt{\pi}\Gamma(-\nu+1)}
	= -\frac{\Gamma(-\nu+1/2)}{\sqrt{\pi}\Gamma(-\nu+1)} < 0.
\end{align*}

Next, let $z = x/y$. Then
\begin{align*}
\lim_{z \to 0^+} x^{2\nu+1} y \widetilde{R}_{\nu}(x,y) & =
	\lim_{z \to 0^+} z^2 x^{-1} y^{-2\nu+3} R_{-\nu}(x,y) - 2\nu \lim_{z \to 0^+} y^{-2\nu+1} K_{-\nu}(x,y) \\
& = -2\nu \frac{\Gamma(-\nu+1/2)}{\sqrt{\pi}\Gamma(-\nu+1)}.
\end{align*}
Observe that the above expression is positive when $\nu < 0$ and negative if $0 < \nu < 1/2$.

In conclusion, we see that (a) and (b) of Lemma \ref{lem:Rlowe} hold.
Item (c) is verified with the aid of the explicit formula for $\widetilde{R}_{1/2}(x,y)$, see Proposition \ref{prop:kere12}.
We have
\begin{align*}
\frac{\pi x^2 y}{\log\frac{x}y} \widetilde{R}_{1/2}(x,y) & = \frac{2}{(1-z^2)\log z}
	+ \frac{\log\frac{z^2}{1-z^2}}{\log z}, \qquad z = \frac{y}x, \\
\pi y^3 \widetilde{R}_{1/2}(x,y) & = \frac{2}{1-z^2} - \frac{1}{z^2} \log\frac{1}{1-z^2}, \qquad z = \frac{x}y.
\end{align*}
Taking the limits of the right-hand sides here as $z \to 0^+$ we get the desired bounds.

It remains to deal with the case $1/2 < \nu < 1$. We will need the following relations that hold for $-1 < \mu < -1/2$
\begin{align}
\lim_{z \to 0^+} x^{2\mu+1} K_{\mu}(x,y) & = 
	\frac{\Gamma(\mu+1/2)}{\sqrt{\pi}\Gamma(\mu+1)},
	\qquad z = \frac{y}x, \label{77}\\
\lim_{z \to 0^+} x^{-2} y^{2\mu+3} K_{\mu}(x,y) & = \frac{\Gamma(\mu+3/2)}{2\sqrt{\pi}\Gamma(\mu+2)}, \qquad z = \frac{x}y. \label{78}
\end{align}
We will prove \eqref{77} and \eqref{78} in a moment.

Let $z = y/x$. Recalling that now $1/2 < \nu < 1$, we have
\begin{align*}
\lim_{z \to 0^+} x^2 y^{2\nu} \widetilde{R}_{\nu}(x,y) & =
	\lim_{z \to 0^+} x^{-2\nu+2} R_{-\nu}(x,y) - 2\nu \lim_{z \to 0^+} x^{-2\nu+1} K_{-\nu}(x,y) \\
& = -\frac{2\Gamma(-\nu+3/2)}{\sqrt{\pi}\Gamma(-\nu+1)} - 2\nu \frac{\Gamma(-\nu+1/2)}{\sqrt{\pi}\Gamma(-\nu+1)} \\
& = - \frac{\Gamma(-\nu+1/2)}{\sqrt{\pi}\Gamma(-\nu+1)} > 0.
\end{align*}

Next, let $z = x/y$. Then
\begin{align*}
\lim_{z\to 0^+} x^{2\nu-1} y^3 \widetilde{R}_{\nu}(x,y) & = 
	\lim_{z\to 0^+} x^{-1} y^{-2\nu+3} R_{-\nu}(x,y) - 2\nu \lim_{z \to 0^+} x^{-2} y^{-2\nu+3} K_{-\nu}(x,y) \\
& = \frac{\Gamma(-\nu+3/2)}{\sqrt{\pi}\Gamma(-\nu+2)} - 2\nu \frac{\Gamma(-\nu+3/2)}{2\sqrt{\pi}\Gamma(-\nu+2)} \\
& = \frac{\Gamma(-\nu+3/2)}{\sqrt{\pi}\Gamma(-\nu+1)} > 0.
\end{align*}

Thus (d) of Lemma \ref{lem:Rlowe} is justified provided that we show \eqref{77} and \eqref{78}.

Considering \eqref{77}, we proceed essentially as in the corresponding computation in the case $\mu > -1/2$. We use
the explicit form of $W_t^{\mu}(x,y)$ and the change of variable $u=2t/x^2$ to get
$$
K_{\mu}(x,y) = \frac{1}{\sqrt{2\pi}} x^{-2\mu-1} \int_0^{\infty}
	\bigg( \Big(\frac{z}{u}\Big)^{-\mu}I_{\mu}\Big(\frac{z}{u}\Big) e^{-\frac{1+z^2}{2u}} - w^{-\mu}I_{\mu}(w)\Big|_{w=0^+}
		e^{-\frac{z^2}{2u}}\bigg) \, \frac{du}{u^{\mu+3/2}},
$$
being $z=y/x$. Denoting by $\mathcal{J}=\mathcal{J}(z)$ the last integral and using the dominated convergence theorem we have
$$
\lim_{z\to 0^+}\mathcal{J}(z) = \frac{1}{2^{\mu}\Gamma(\mu+1)} \int_0^{\infty}\Big( e^{-\frac{1}{2u}}-1\Big) \, \frac{du}{u^{\mu+3/2}}.
$$
The integral occurring here can be computed. Indeed, changing the variable $t=1/(2u)$ and using Tonelli's theorem we get
\begin{align*}
\int_0^{\infty}\Big( e^{-\frac{1}{2u}}-1\Big) \, \frac{du}{u^{\mu+3/2}}
& = 2^{\mu+1/2}\int_0^{\infty}\big( e^{-t}-1\big) t^{\mu-1/2}\, dt \\
& = - 2^{\mu+1/2} \int_0^{\infty} \int_0^t e^{-s}\, ds \, t^{\mu-1/2}\, dt \\
& = - 2^{\mu+1/2} \int_0^{\infty} e^{-s} \int_s^{\infty} t^{\mu-1/2}\, dt \, ds \\
& = \frac{2^{\mu+1/2}}{\mu+1/2} \int_0^{\infty} e^{-s} s^{\mu+1/2}\, ds = 2^{\mu+1/2}\Gamma(\mu+1/2).
\end{align*}
From this \eqref{77} follows.

Passing to \eqref{78}, we again use the explicit form of $W_t^{\mu}(x,y)$, but now change the variable $u=2t/y^2$ and get
$$
K_{\mu}(x,y) = \frac{1}{\sqrt{2\pi}} x^2 y^{-2\mu-3} \int_0^{\infty}
	\frac{\Big(\frac{z}{u}\Big)^{-\mu}I_{\mu}\Big(\frac{z}{u}\Big) e^{-\frac{1+z^2}{2u}} - w^{-\mu}I_{\mu}(w)\Big|_{w=0^+}
		e^{-\frac{1}{2u}}}{\Big(\frac{z}{u}\Big)^2} \, \frac{du}{u^{\mu+7/2}},
$$
where $z=x/y$.
Decompose the last integral as $\mathcal{I}_1(z)+\mathcal{I}_2(z)$, where
\begin{align*}
\mathcal{I}_1(z) & = \int_0^{\infty}
	\frac{\Big(\frac{z}{u}\Big)^{-\mu}I_{\mu}\Big(\frac{z}{u}\Big) e^{-\frac{1+z^2}{2u}} - w^{-\mu}I_{\mu}(w)\Big|_{w=0^+}
		e^{-\frac{1+z^2}{2u}}}{\Big(\frac{z}{u}\Big)^2} \, \frac{du}{u^{\mu+7/2}}, \\
\mathcal{I}_2(z) & = \int_0^{\infty}
	\frac{w^{-\mu}I_{\mu}(w)\Big|_{w=0^+} e^{-\frac{1+z^2}{2u}} - w^{-\mu}I_{\mu}(w)\Big|_{w=0^+}
		e^{-\frac{1}{2u}}}{\Big(\frac{z}{u}\Big)^2} \, \frac{du}{u^{\mu+7/2}}.
\end{align*}

From the series expansion \eqref{bes:ser}
it is seen that the integrand in $\mathcal{I}_1(z)$ tends to $(2^{\mu+2}\Gamma(\mu+2))^{-1}e^{-1/2u}u^{-\mu-7/2}$ as
$z \to 0^+$. Therefore, by the dominated convergence theorem,
$$
\lim_{z\to 0^+}\mathcal{I}_1(z) = \frac{1}{2^{\mu+2}\Gamma(\mu+2)} \int_0^{\infty} e^{-\frac{1}{2u}}\, \frac{du}{u^{\mu+7/2}}
	= \frac{\sqrt{2}\Gamma(\mu+5/2)}{\Gamma(\mu+2)}.
$$

As for $\mathcal{I}_2(z)$, we have
$$
\mathcal{I}_2(z) = \frac{1}2 w^{-\mu}I_{\mu}(w)\Big|_{w=0^+} \int_0^{\infty} e^{-\frac{1}{2u}}
	\frac{e^{-\frac{z^2}{2u}}-1}{\frac{z^2}{2u}} \, \frac{du}{u^{\mu+5/2}}
$$
and here the integrand converges to $-e^{-1/2u}u^{-\mu-5/2}$ as $z \to 0^+$.
Consequently, by the dominated convergence theorem we get
$$
\lim_{z\to 0^+} \mathcal{I}_2(z) = \frac{-1}{2^{\mu+1}\Gamma(\mu+1)} \int_0^{\infty} e^{-\frac{1}{2u}} \frac{du}{u^{\mu+5/2}}
	= - \frac{\sqrt{2}\Gamma(\mu+3/2)}{\Gamma(\mu+1)}.
$$
The formulas obtained for the limits of $\mathcal{I}_1(z)$ and $\mathcal{I}_2(z)$ imply \eqref{78}.

The proof of Lemma \ref{lem:Rlowe} is complete.
\end{proof}

\begin{proof}[{Proof of Theorem \ref{thm:Rexo}, sufficiency part}]
We assume that $0 \neq \nu < 1/2$ and focus first on the operator $\widetilde{R}_{\nu}$. We have, see \eqref{37},
\begin{align*}
\widetilde{R}_{\nu}f(x) & = x^{-2\nu} R_{-\nu}\big( y^{2\nu}f(y)\big)(x) - \frac{2\nu}{x} x^{-2\nu}
	\big( B_{-\nu}^{\textrm{cls}}\big)^{-1/2}\big( y^{2\nu}f(y)\big)(x) \\
& \equiv \widetilde{R}_{\nu}^1 f(x) - \widetilde{R}_{\nu}^2 f(x).
\end{align*}
We will consider each of the two component operators separately.

The kernel of $\widetilde{R}_{\nu}^1$ is $\widetilde{R}_{\nu}^1(x,y) = (xy)^{-2\nu}R_{-\nu}(x,y)$.
In view of \cite[Lemma 4.3]{BHNV},
$$
\widetilde{R}_{\nu}^1(x,y) = \frac{1}{\pi} \frac{(xy)^{-\nu-1/2}}{y-x}
	+ \mathcal{O}\bigg( y^{-2\nu-2} \Big(1+\log\frac{xy}{(y-x)^2}\Big)\bigg), \qquad x/2 < y < 2x,
$$
and, moreover, in the off-diagonal region
$$
\big|\widetilde{R}_{\nu}^1(x,y)\big| \lesssim
	\begin{cases}
		x^{-2} y^{-2\nu}, & y \le x/2, \\
		x^{-2\nu+1} y^{-3}, & 2x \le y.
	\end{cases}
$$
Therefore,
$$
\big|\widetilde{R}_{\nu}^1f(x)\big|
	\lesssim |\mathcal{H}_{\nu,\loc}f(x)| + N^{\log}|f|(x) + H_0^1|f|(x) + H_{\infty}^{-2\nu+1}|f|(x).
$$

Next, we look at the operator $\widetilde{R}_{\nu}^2$. Its kernel is
$\widetilde{R}_{\nu}^2(x,y) = 2\nu x^{-1} (xy)^{-2\nu} K_{-\nu}(x,y)$.
From \cite[Theorem 2.1]{NoSt1} we have that for $\mu > -1/2$
$$
K_{\mu}(x,y) \simeq (x+y)^{-2\mu-1}\log\frac{2(x+y)}{|x-y|} \simeq
	\begin{cases}
		x^{-2\mu-1}, & y \le x/2, \\
		(xy)^{-\mu-1/2}\log\frac{xy}{(y-x)^2}, & x/2 < y < 2x, \\
		y^{-2\mu-1}, & 2x \le y.
	\end{cases}
$$
This implies the control
$$
\big|\widetilde{R}_{\nu}^2 f(x) \big| \lesssim N^{\log}|f|(x) + H_0^1|f|(x) + H_{\infty}^{-2\nu-1}|f|(x).
$$

Observing now that $H_{\infty}^{-2\nu+1}|f|$ is dominated by $H_{\infty}^{-2\nu-1}|f|$, we get the final bound
$$
\big|\widetilde{R}_{\nu} f(x) \big| \lesssim 
	|\mathcal{H}_{\nu,\loc}f(x)| + N^{\log}|f|(x) + H_0^1|f|(x) + H_{\infty}^{-2\nu-1}|f|(x).
$$
Taking into account Lemmas \ref{lem:MH}, \ref{lem:NN}, \ref{lem:H0} and \ref{lem:Hinf}, we see that
$\widetilde{R}_{\nu}$ is bounded on $L^p(x^{\delta}dx)$ when
$1< p < \infty$ and $(2\nu+1)p - 1 < \delta < 2p-1$, it is bounded from $L^1(x^{\delta}dx)$ to weak
$L^1(x^{\delta}dx)$ when $2\nu \le \delta \le 1$ (with the first inequality weakened in case $\nu=-1/2$),
and it is of restricted weak type $(p,p)$ with respect to the measure space $(\mathbb{R}_+,x^{\delta}dx)$ when $1< p < \infty$
and $(2\nu+1)p-1 \le \delta \le 2p-1$ (with the first inequality weakened in case $\nu=-1/2$).
Hence sufficiency parts of items (a1), (b1) and (c1) are verified.

Treatment of the adjoint $\widetilde{R}_{\nu}^*$ is similar. We can split
$\widetilde{R}_{\nu}^* = (\widetilde{R}_{\nu}^1)^* - (\widetilde{R}_{\nu}^2)^*$, and the kernels of the two component
operators are $\widetilde{R}_{\nu}^1(y,x)$ and $\widetilde{R}_{\nu}^2(y,x)$, respectively.
Using this and having in mind our previous considerations we arrive at the control
\begin{align*}
\big| \big(\widetilde{R}_{\nu}^1\big)^* f(x)\big| & \lesssim
	 |\mathcal{H}_{\nu,\loc}f(x)| + N^{\log}|f|(x) + H_0^2|f|(x) + H_{\infty}^{-2\nu}|f|(x), \\
\big| \big(\widetilde{R}_{\nu}^2\big)^* f(x)\big| & \lesssim
	 N^{\log}|f|(x) + H_0^0|f|(x) + H_{\infty}^{-2\nu}|f|(x),
\end{align*}
and this implies
$$
\big| \big(\widetilde{R}_{\nu}\big)^* f(x)\big| \lesssim
	 |\mathcal{H}_{\nu,\loc}f(x)| + N^{\log}|f|(x) + H_0^0|f|(x) + H_{\infty}^{-2\nu}|f|(x).
$$

In view of Lemmas \ref{lem:MH}, \ref{lem:NN}, \ref{lem:H0} and \ref{lem:Hinf} we conclude the following mapping properties.
$\widetilde{R}_{\nu}^*$ is bounded on $L^p(x^{\delta}dx)$ when
$1< p < \infty$ and $2\nu p - 1 < \delta < p-1$, it is bounded from $L^1(x^{\delta}dx)$ to weak
$L^1(x^{\delta}dx)$ when $2\nu-1 \le \delta \le 0$, and it is of restricted weak type $(p,p)$ with
respect to the measure space $(\mathbb{R}_+,x^{\delta}dx)$ when $1 < p < \infty$ and $2\nu p-1 \le \delta \le p-1$.
This gives sufficiency parts of items (a2), (b2) and (c2) of the theorem.
\end{proof}

\begin{proof}[{Proof of Theorem \ref{thm:Rexo}, necessity part}]
In this proof we always assume that $1 \le p < \infty$ and $0 \neq \nu < 1/2$.

Consider first the operator $\widetilde{R}_{\nu}$.
It is sufficient we verify the following claims.
\begin{itemize}
\item[(A)]
$\widetilde{R}_{\nu}$ is not bounded on $L^1(x^{\delta}dx)$ for $2\nu \le \delta \le 1$
(with the first inequality strictened if $\nu = -1/2$).
\item[(B)]
$\widetilde{R}_{\nu}$ is not bounded from $L^p(x^{\delta}dx)$ to weak $L^p(x^{\delta}dx)$
when $p > 1$ and either $\delta = (2\nu+1)p-1$ or $\delta = 2p-1$.
\item[(C)]
$\widetilde{R}_{\nu}$ is not of restricted weak type $(p,p)$ with respect to $(\mathbb{R}_+,x^{\delta}dx)$
if either $\delta < (2\nu+1)p-1$, with the inequality weakened when $\nu = -1/2$, or $\delta > 2p-1$.
\end{itemize}

To prove (A) we need the estimate in the diagonal region
\begin{equation} \label{dig}
\widetilde{R}_{\nu}(x,y) = \frac{1}{\pi} \frac{(xy)^{-\nu-1/2}}{y-x}
	+ \mathcal{O}\bigg( y^{-2\nu-2} \Big(1+\log\frac{xy}{(y-x)^2}\Big)\bigg), \qquad x/2 < y < 2x,
\end{equation}
that can be deduced from the proof of the sufficiency part in Theorem \ref{thm:Rexo}.
Following the counterexample from (C) of the proof of \cite[Theorem 2.3]{BHNV}, take
$f_{\epsilon} = \chi_{(1,1+\epsilon)}$ for $0 < \epsilon < 1/4$. Then, in view of \eqref{dig},
\begin{align*}
\big\| \widetilde{R}_{\nu}f_{\epsilon}\big\|_{L^1(x^{\delta}dx)} & \ge \int_{1+2\epsilon}^{2}
	\bigg| \int_{x/2}^{2x} \widetilde{R}_{\nu}(x,y) f_{\epsilon}(y)\, d\eta_{\nu}(y)\bigg| x^{\delta}\, dx \\
& \gtrsim \int_{1+2\epsilon}^2 \bigg| \int_{1}^{1+\epsilon} \frac{dy}{y-x}\bigg| \, dx
	- c \int_{1+2\epsilon}^2 N^{\log}f_{\epsilon}(x) x^{\delta}\, dx,
\end{align*}
with a constant $c$ independent of $\epsilon$. If $\widetilde{R}_{\nu}$ were bounded on $L^1(x^{\delta}dx)$ we would have,
see Lemma~\ref{lem:NN},
$$
\int_{1+2\epsilon}^2 \bigg| \int_{1}^{1+\epsilon} \frac{dy}{y-x}\bigg| \, dx
	\lesssim \big\| \widetilde{R}_{\nu}f_{\epsilon}\big\|_{L^1(x^{\delta}dx)} +
	\big\| N^{\log}f_{\epsilon}\big\|_{L^1(x^{\delta}dx)} \lesssim \|f_{\epsilon}\|_{L^1(x^{\delta}dx)} \simeq \epsilon.
$$
But this is not true as $\epsilon \to 0^+$, since
$$
\int_{1+2\epsilon}^2 \bigg| \int_{1}^{1+\epsilon} \frac{dy}{y-x}\bigg| \, dx \simeq
	\int_{1+2\epsilon}^2 \log\frac{x-1}{x-1-\epsilon}\, dx \simeq \epsilon \log\frac{1}{\epsilon},
		\qquad 0 < \epsilon < \frac{1}{4}.
$$
This counterexample actually shows that $\widetilde{R}_{\nu}$ is not bounded on $L^1(x^{\delta}dx)$
for any $\delta \in \mathbb{R}$.

Next, we show (B). To this end we assume that $p>1$. Consider first $\delta = 2p-1$.
Split $\widetilde{R}_{\nu}$ as
$$
\widetilde{R}_{\nu}f(x) = \widetilde{R}_{\nu}\big( \chi_{\{0 < y < x/b\}}f(y)\big)(x) + 
	\widetilde{R}_{\nu}\big( \chi_{\{x/b \le y\}}f(y)\big)(x),
$$
where $b$ is the constant from Lemma \ref{lem:Rlowe}; we can assume that $b > 2$.
Then, by Lemma \ref{lem:Rlowe},
\begin{equation} \label{len}
\Big| \widetilde{R}_{\nu}\big( \chi_{\{0 < y < x/b\}}f(y)\big)(x) \Big| \gtrsim \big| H_{0,b}^1 f(x)\big|, \qquad x > 0,
\end{equation}
and, moreover,
\begin{align*}
& \Big| \widetilde{R}_{\nu}\big( \chi_{\{x/b \le y\}}f(y)\big)(x) \Big| \\ & \quad \lesssim \big( H_0^1-H_{0,b}^1\big)|f|(x)
	+ H_{\infty}^{-2\nu-1}|f|(x) + |\mathcal{H}_{\nu,\loc}f(x)| + N^{\log}|f|(x), \qquad x > 0.
\end{align*}
By Proposition \ref{prop:Hb} and Lemmas \ref{lem:Hinf}, \ref{lem:MH} and \ref{lem:NN} we see that the right-hand side
here satisfies weak type $(p,p)$ inequality with respect to $(\mathbb{R}_+,x^{2p-1}dx)$.
But the right-hand side in \eqref{len} does not, in view of the necessity part of Lemma \ref{lem:H0}(b) and the
comment preceding Proposition \ref{prop:Hb}. Altogether, this shows that $\widetilde{R}_{\nu}$ is not weak
type $(p,p)$ with respect to $(\mathbb{R}_+,x^{2p-1}dx)$.

Treatment of $\delta = (2\nu+1)p-1$ is similar. One splits
$$
\widetilde{R}_{\nu}f(x) = \widetilde{R}_{\nu}\big( \chi_{\{bx < y\}}f(y)\big)(x) + 
	\widetilde{R}_{\nu}\big( \chi_{\{y \le bx\}}f(y)\big)(x)
$$
and then uses Lemma \ref{lem:Rlowe} with $b > 2$ to get
\begin{equation} \label{len2}
\Big| \widetilde{R}_{\nu}\big( \chi_{\{bx < y\}}f(y)\big)(x) \Big| \gtrsim \big| H_{\infty,b}^{-2\nu-1} f(x)\big|, \qquad x > 0.
\end{equation}
Furthermore,
\begin{align*}
& \Big| \widetilde{R}_{\nu}\big( \chi_{\{y \le bx\}}f(y)\big)(x) \Big| \\
& \quad \lesssim \big( H_{\infty}^{-2\nu-1}-H_{\infty,b}^{-2\nu-1}\big)|f|(x)
	+ H_{0}^{1}|f|(x) + |\mathcal{H}_{\nu,\loc}f(x)| + N^{\log}|f|(x), \qquad x > 0.
\end{align*}
The right-hand side here is of weak type $(p,p)$ with respect to $(\mathbb{R}_+,x^{(2\nu+1)p-1}dx)$, by Proposition
\ref{prop:Hb} and Lemmas \ref{lem:H0}, \ref{lem:MH}, \ref{lem:NN}. But the right-hand side in \eqref{len2} is not,
due to the necessity part of Lemma \ref{lem:Hinf}(b), see the comment preceding Proposition \ref{prop:Hb}.
It follows that $\widetilde{R}_{\nu}$ is not of weak type $(p,p)$ with respect to $(\mathbb{R}_+,x^{(2\nu+1)p-1}dx)$.

Finally, we pass to (C). Here the line of reasoning is exactly the same as in case of (B).
We use the same decompositions of $\widetilde{R}_{\nu}$ and argue with the aid of Proposition \ref{prop:Hb}
and Lemmas \ref{lem:H0}(c), \ref{lem:Hinf}(c), \ref{lem:MH} and \ref{lem:NN}.
The details are straightforward and thus omitted.

Analysis of $\widetilde{R}_{\nu}^{*}$ is completely parallel and the details are left to the reader.
We only mention that the relevant Hardy operators are $H_0^0$ and $H_{\infty}^{-2\nu}$.
\end{proof}

\begin{pro} \label{pro:restl12}
Let $1/2 < \nu < 1$ be fixed. Then
$$
\widetilde{R}_{\nu}(x,y) = \frac{1}{\pi} \frac{(xy)^{-\nu-1/2}}{y-x}
	+ \mathcal{O}\bigg( y^{-2\nu-2} \Big(1+\log\frac{xy}{(y-x)^2}\Big)\bigg), \qquad x/2 < y < 2x.
$$
Moreover, in the off-diagonal region
$$
\big|\widetilde{R}_{\nu}(x,y)\big| \lesssim
	\begin{cases}
		x^{-2} y^{-2\nu}, & y \le x/2, \\
		x^{-2\nu+1} y^{-3}, & 2x \le y.
	\end{cases}
$$
\end{pro}

\begin{proof}
In view of the proof of Lemma \ref{lem:Rlowe}(d) one can assume that $x \simeq y$.
Further, by the decomposition $\widetilde{R}_{\nu}(x,y) = \widetilde{R}_{\nu}^1(x,y) - \widetilde{R}_{\nu}^2(x,y)$ from
the proof of the sufficiency part of Theorem \ref{thm:Rexo} and the bounds for $\widetilde{R}_{\nu}^1$ stated there,
it is enough to bound suitably $\widetilde{R}_{\nu}^2(x,y)$ for comparable $x$ and $y$. This further boils down
to proving the following bound for the classical compensated potential kernel and $-1 < \mu < -1/2$:
for any given $b > 1$,
\begin{equation} \label{Kmu}
|K_{\mu}(x,y)| \lesssim y^{-2\mu-1} \bigg( 1 + \log^+\frac{xy}{(y-x)^2}\bigg), \qquad x/b < y < bx,
\end{equation}
with $\log^+$ standing for the positive part of the logarithm.

Recall that for $-1 < \mu < -1/2$,
$$
K_{\mu}(x,y) = \frac{1}{\sqrt{\pi}} \int_0^{\infty} \Big( W_t^{\mu}(x,y) - W_t^{\mu}(0,y)\Big)\, \frac{dt}{\sqrt{t}}.
$$
We split the integral here and write
$$
K_{\mu}(x,y) = \frac{1}{\sqrt{\pi}} \Bigg\{ \int_0^{xy} + \int_{xy}^{\infty} \Bigg\} \Big(\ldots\Big)\, \frac{dt}{\sqrt{t}}
	\equiv K_{\mu}^1(x,y) + K_{\mu}^2(x,y).
$$
We will estimate separately each of the two component kernels. To this end we always assume that $b>1$ is fixed and
$x/b < y < bx$.

In order to bound $K_{\mu}^2(x,y)$ we first observe that, for $xy < t$,
\begin{align*}
W_t^{\mu}(x,y) - W_t^{\mu}(0,y) & =
	\frac{1}{(2t)^{\mu+1}} e^{-\frac{y^2}{4t}} \bigg[ \Big(\frac{xy}{2t}\Big)^{-\mu} I_{\mu}\Big(\frac{xy}{2t}\Big)
		e^{-\frac{x^2}{4t}} - w^{-\mu}I_{\mu}(w)\Big|_{w=0^+} \bigg] \\
& = \frac{1}{(2t)^{\mu+1}} e^{-\frac{y^2}{4t}}\bigg[ w^{-\mu}I_{\mu}(w)\Big|_{w=0^+} \Big(e^{-\frac{x^2}{4t}}-1\Big)
	+ \mathcal{O}\bigg( \Big(\frac{xy}t\Big)^2\bigg) e^{-\frac{x^2}{4t}} \bigg] \\
& = \frac{1}{t^{\mu+1}} e^{-\frac{y^2}{4t}}\mathcal{O}\bigg( \frac{x^2}{t}\bigg)
		+ \frac{1}{t^{\mu+1}} e^{-\frac{y^2}{4t}} \mathcal{O}\bigg( \Big(\frac{xy}t\Big)^2\bigg),
\end{align*}
where we used \eqref{bes:ser}. Therefore, (recall that we are considering $x \simeq y$)
$$
\big| W_t^{\mu}(x,y) - W_t^{\mu}(0,y)\big| \lesssim \frac{1}{t^{\mu+1}} e^{-\frac{y^2}{4t}}
	\bigg(\frac{y^2}{t}+ \frac{y^4}{t^2}\bigg)
	 \simeq \frac{y^2}{t^{\mu+2}} e^{-\frac{y^2}{4t}}, \qquad xy < t.
$$
Consequently,
$$
\big|K_{\mu}^2(x,y)\big| \lesssim y^2 \int_{xy}^{\infty} e^{-\frac{y^2}{4t}}\, \frac{dt}{t^{\mu+5/2}}
	\simeq y^{-2\mu-1},
$$
where the last relation follows by the change of variable $u=y^2/4t$.

To estimate $K_{\mu}^1(x,y)$ we write
$$
\big| K_{\mu}^1(x,y) \big| \le \int_0^{xy} W_t^{\mu}(x,y)\, \frac{dt}{\sqrt{t}} + \int_0^{xy} W_t^{\mu}(0,y)\, \frac{dt}{\sqrt{t}}
	\equiv K_{\mu}^3(x,y) + K_{\mu}^4(x,y).
$$
Then
$$
K_{\mu}^4(x,y) = \frac{1}{2^{\mu}\Gamma(\mu+1)} \int_0^{xy} \frac{1}{(2t)^{\mu+1}} e^{-\frac{y^2}{4t}}\, \frac{dt}{\sqrt{t}}
	\simeq y^{-2\mu-1},
$$
by changing the variable $u=y^2/4t$ and since $x \simeq y$.

Finally, to deal with $K_{\mu}^3(x,y)$ we use the large argument asymptotic behavior \eqref{bes:est}
and get
$$
K_{\mu}^3(x,y) \lesssim (xy)^{-\mu-1/2} \int_0^{xy} e^{-\frac{(x-y)^2}{4t}}\, \frac{dt}{\sqrt{t}}
	\simeq (xy)^{-\mu-1/2} \int_{\frac{(x-y)^2}{4xy}}^{\infty} e^{-u}\, \frac{du}{u}.
$$
The last integral can easily be estimated, so taking also into account that $x \simeq y$ we conclude
$$
K_{\mu}^3(x,y) \lesssim y^{-2\mu-1}\bigg( 1 + \log^+ \frac{xy}{(x-y)^2}\bigg).
$$
Now \eqref{Kmu} follows and this finishes the proof.
\end{proof}

\begin{proof}[Proof of Theorem \ref{thm:Rexob}]
The proof goes along the lines of the proof of Theorem \ref{thm:Rexo}.
Let us consider first the main case $1/2 < \nu < 1$.

To show the sufficiency part, use Proposition \ref{pro:restl12} to obtain the control
\begin{align*}
\big| \widetilde{R}_{\nu}f(x)\big| &
	\lesssim |\mathcal{H}_{\nu,\loc}f(x)| + N^{\log}|f|(x) + H_0^1 |f|(x) + H_{\infty}^{-2\nu+1}|f|(x), \\
\big| \widetilde{R}^{*}_{\nu}f(x)\big| &
	\lesssim |\mathcal{H}_{\nu,\loc}f(x)| + N^{\log}|f|(x) + H_0^2 |f|(x) + H_{\infty}^{-2\nu}|f|(x).
\end{align*}
Then apply Lemmas \ref{lem:MH}, \ref{lem:NN}, \ref{lem:H0} and \ref{lem:Hinf} to get the conclusion.

The necessity part is proved by means of Proposition \ref{pro:restl12},
Lemma \ref{lem:Rlowe}(d) and necessity results for the Hardy operators
involved; see Lemmas \ref{lem:H0} and \ref{lem:Hinf}.

The case $\nu = 1/2$ is slightly different since more logarithms come into play. By Proposition~\ref{prop:keru12}
we have the control
\begin{align*}
\big| \widetilde{R}_{1/2}f(x)\big| &
	\lesssim |\mathcal{H}_{1/2,\loc}f(x)| + N^{\log}|f|(x) + H_0^{1,\log} |f|(x) + H_{\infty}^{0}|f|(x), \\
\big| \widetilde{R}^{*}_{1/2}f(x)\big| &
	\lesssim |\mathcal{H}_{1/2,\loc}f(x)| + N^{\log}|f|(x) + H_0^2 |f|(x) + H_{\infty}^{-1,\log}|f|(x).
\end{align*}
Then, to show the sufficiency part, one uses Lemma \ref{lem:Hlog} together with the other lemmas mentioned before.
The necessity part is deduced with the aid of Proposition \ref{prop:keru12},
Lemma \ref{lem:Rlowe}(c) and (implicit) necessity parts in Lemma \ref{lem:Hlog}.
\end{proof}

%%%%%%%%%%%%%%%%%%%%%%%%%%%%%%%%%%%%%%%%%%%%%%%%%%%%%%%%%%%%%%
%%%%%%%%%%%%%%%%%%%%%%%%%%%%%%%%%%%%%%%%%%%%%%%%%%%%%%%%%%%%%%

\section{Vertical $g$-function} \label{sec:g}

Recall that the heat semigroup based vertical $g$-function in the one-dimensional classical Bessel setting is defined as
$$
g_{\nu}(f)(x) = \bigg\| \frac{\partial}{\partial t} W_t^{\nu}f(x)\bigg\|_{L^2(\mathbb{R}_+,t dt)}, \qquad x > 0.
$$
Here $\nu > -1$ is in the classical range. We define the exotic counterpart of $g_{\nu}$ in the natural way,
$$
\widetilde{g}_{\nu}(f)(x) =
	\bigg\| \frac{\partial}{\partial t} \widetilde{W}_t^{\nu}f(x)\bigg\|_{L^2(\mathbb{R}_+,t dt)}, \qquad x > 0,
$$
where $0 \neq \nu < 1$ is in the exotic range. Observe that, in view of \eqref{Wexotocls},
\begin{equation} \label{g_exotocls}
\widetilde{g}_{\nu}(f)(x) = x^{-2\nu} g_{-\nu}\big(y^{2\nu}f\big)(x), \qquad x > 0, \quad 0 \neq \nu < 1.
\end{equation}

The following characterization of mapping properties of $g_{\nu}$ was obtained in \cite[Theorem 2.5]{BHNV}.
\begin{thm}[{\cite{BHNV}}] \label{thm:gcls}
Let $\nu > -1$, $1 \le p < \infty$, $\delta \in \mathbb{R}$. Then the $g$-function $g_{\nu}$,
considered on the measure space $(\mathbb{R}_+,x^{\delta}dx)$, has the following mapping properties:
\begin{itemize}
\item[(a)]
$g_{\nu}$ is of strong type $(p,p)$ if and only if $p > 1$ and $-1 < \delta < (2\nu+2)p -1$;
\item[(b)]
$g_{\nu}$ is of weak type $(p,p)$ if and only if $-1 < \delta < (2\nu+2)p-1$, with the second inequality weakened in case $p=1$;
\item[(c)]
$g_{\nu}$ is of restricted weak type $(p,p)$ if and only if $-1 < \delta \le (2\nu+2)p-1$.
\end{itemize}
\end{thm}

We will prove a similar result in the exotic case.
\begin{thm} \label{thm:gexo}
Let $0 \neq \nu < 1$, $1 \le p < \infty$, $\delta \in \mathbb{R}$. Then the $g$-function $\widetilde{g}_{\nu}$,
considered on the measure space $(\mathbb{R}_+,x^{\delta}dx)$, has the following mapping properties:
\begin{itemize}
\item[(a)]
$\widetilde{g}_{\nu}$ is of strong type $(p,p)$ if and only if $p > 1$ and $2\nu p - 1 < \delta < 2p -1$;
\item[(b)]
$\widetilde{g}_{\nu}$ is of weak type $(p,p)$ if $2\nu p -1 < \delta < 2p-1$, with both inequalities
	weakened in case $p=1$; otherwise
	$\widetilde{g}_{\nu}$ is not of weak type $(p,p)$, excluding possibly 
	the case when $\delta = 2\nu p -1$ and $p>1$.
\item[(c)]
$\widetilde{g}_{\nu}$ is of restricted weak type $(p,p)$ if and only if $2\nu p-1 \le \delta \le 2p-1$.
\end{itemize}
\end{thm}

Unfortunately, we were not able to obtain a full characterization for the weak type $(p,p)$.
More precisely, the question whether $\widetilde{g}_{\nu}$ is of weak type $(p,p)$ for $\delta=2\nu p-1$, $1<p<\infty$,
$\nu \neq 0$, remains open. Roughly speaking, the answer seems to require dealing with some oscillations which are
hard to grasp.

Assuming $p<\infty$, $g_{\nu}$ has the same mapping properties as $W_{*}^{\nu}$, see Theorems \ref{thm:maxWcls} and
\ref{thm:gcls}. This is also true about $\widetilde{g}_{\nu}$ and $\widetilde{W}_{*}^{\nu}$, see Theorems
\ref{thm:maxWexo} and \ref{thm:gexo}, up to the endpoint issue that remains to be sorted out.

To prove Theorem \ref{thm:gexo} we adopt the same strategy as in case of the heat maximal operator $\widetilde{W}_{*}^{\nu}$.

\begin{proof}[{Proof of Theorem \ref{thm:gexo}}]
From the proof of \cite[Theorem 2.5]{BHNV} we have the control
$$
g_{\nu}(f)(x) \lesssim H_0^{2\nu +1}|f|(x) + H_{\infty}^0|f|(x) + N|f|(x) + \mathfrak{g}_{\nu,\loc}(f)(x),
	\qquad x > 0, \quad \nu > -1.
$$
Combining this with \eqref{g_exotocls}, we get
\begin{equation} \label{gfcontr}
\widetilde{g}_{\nu}(f)(x) \lesssim H_0^1|f|(x) + H_{\infty}^{-2\nu}|f|(x) + N|f|(x)
		+ \mathfrak{g}_{\nu,\loc}(f)(x), \qquad x > 0, \quad 0 \neq \nu < 1.
\end{equation}
Appealing now to Lemmas \ref{lem:H0}, \ref{lem:Hinf}, \ref{lem:NN} and \ref{lem:MH} we get the sufficiency part in
Theorem \ref{thm:gexo}.

It remains to prove the necessity part. To this end we always assume $0 \neq \nu < 1$ and $1 \le p < \infty$,
and the underlying space is always $(\mathbb{R}_+,x^{\delta}dx)$. It is enough we prove the following statements.
\begin{itemize}
\item[(A)] If $\widetilde{g}_{\nu}$ is of restricted weak type $(p,p)$, then $2\nu p -1 \le \delta \le 2p -1$.
\item[(B)] $\widetilde{g}_{\nu}$ is not of weak type $(p,p)$ when $p>1$ and $\delta = 2p -1$.
\item[(C)] $\widetilde{g}_{\nu}$ is not of strong type $(1,1)$ if $2\nu-1 \le \delta \le 1$.
\end{itemize}
Below we restrict to $f \ge 0$.

From \eqref{exoWcls} and \cite[Lemma 5.2]{BHNV} it follows that there exists a constant $a>0$ (depending only on $\nu$)
such that
\begin{equation} \label{ptlow}
\frac{\partial}{\partial t} \widetilde{W}_t^{\nu}(x,y) \lesssim \frac{- t^{\nu-2}}{(xy)^{2\nu}} \qquad \textrm{if}\quad
0 < x,y < a \;\; \textrm{and} \;\; t \ge 1 \quad \textrm{or} \quad 0 < y < x \;\; \textrm{and}\;\; \frac{x^2}t \le a.
\end{equation}
Further, for sufficiently nice $f$
$$
\widetilde{g}_{\nu}(f)(x) = \bigg\| \int_0^{\infty} \frac{\partial}{\partial t} \widetilde{W}_t^{\nu}(x,y) f(y)\,
		d\eta_{\nu}(y) \bigg\|_{L^2(\mathbb{R}_+, t dt)}.
$$

Assume that $\support f \subset [0,1]$ and consider $x > 1$. Then, with the aid of \eqref{ptlow}, we can estimate
\begin{align} \nonumber
\widetilde{g}_{\nu}(f)(x) & \ge \bigg\| \int_0^1 \frac{\partial}{\partial t} \widetilde{W}_t^{\nu}(x,y) f(y)\, d\eta_{\nu}(y)
		\bigg\|_{L^2((x^2/a,\infty),t dt)} \\ \label{gb1}
& \gtrsim x^{-2\nu} \big\| t^{\nu-2}\big\|_{L^2((x^2/a,\infty),t dt)} \int_0^1 f(y) y \, dy \\
& \simeq x^{-2} \int_0^1 f(y) y \, dy. \nonumber
\end{align}

Next, assume that $\support f \subset [0,a]$ and consider $x < a$. Then, in view of \eqref{ptlow}, we get
\begin{align} \nonumber
\widetilde{g}_{\nu}(f)(x) & \ge \bigg\| \int_0^a \frac{\partial}{\partial t} \widetilde{W}_t^{\nu}(x,y) f(y)\, d\eta_{\nu}(y)
		\bigg\|_{L^2((1,\infty),t dt)} \\ \label{gb2}
& \gtrsim x^{-2\nu} \big\| t^{\nu-2}\big\|_{L^2((1,\infty),t dt)} \int_0^a f(y) y \, dy \\
& \simeq x^{-2\nu} \int_0^a f(y) y \, dy. \nonumber
\end{align}

First, we prove (A). Choosing $f = \chi_{(1/2,1)}$ and using \eqref{gb1} we get the bound
$$
\widetilde{g}_{\nu}(f)(x) \gtrsim x^{-2}, \qquad x > 1.
$$
Proceeding now as in the proof of Theorem \ref{thm:maxWexo}, see \eqref{55exo}, we get the upper bound for $\delta$ in (A).
On the other hand, choosing $f = \chi_{(a/2,a)}$ and using \eqref{gb2} we obtain
$$
\widetilde{g}_{\nu}(f)(x) \gtrsim x^{-2\nu}, \qquad 0 < x < a.
$$
From here, by the arguments from the above mentioned proof, see \eqref{6exo}, we arrive at the lower bound for $\delta$ in (A).

To verify (B), observe that it follows by \eqref{gb1} and the corresponding argument from the proof
of Theorem \ref{thm:maxWexo}.

Finally, we deal with (C). Here we could also provide a counterexample, but we prefer to argue in a shorter way, though less directly.
In \cite[p.\,134--136]{BHNV} it was proved that $g_{\nu}$, $\nu > -1$, is not bounded on $L^1(x^{\delta}dx)$ for any
$\delta \in \mathbb{R}$.
Taking into account \eqref{g_exotocls} and Proposition \ref{thm:trans}, we conclude that $\widetilde{g}_{\nu}$,
$0 \neq \nu < 1$, is not bounded on $L^1(x^{\delta}dx)$ for any $\delta \in \mathbb{R}$.

\end{proof}

%%%%%%%%%%%%%%%%%%%%%%%%%%%%%%%%%%%%%%%%%%%%%%%%%%%%%%%%%%%%%%
%%%%%%%%%%%%%%%%%%%%%%%%%%%%%%%%%%%%%%%%%%%%%%%%%%%%%%%%%%%%%%

\section{Fractional integrals} \label{sec:frac}

Let $I^{\nu,\sigma}$ be the fractional integral (Riesz potential) of order
$\sigma > 0$ in the classical one-dimensional Bessel setting of type $\nu > -1$.
We have the integral representation (see e.g.\ \cite[Section 2.1]{NoSt1})
$$
I^{\nu,\sigma}f(x) = \int_0^{\infty} K^{\nu,\sigma}(x,y) f(y) \, d\eta_{\nu}(y), \qquad x > 0,
$$
where the integral kernel (potential kernel) expresses via the Bessel heat kernel,
$$
K^{\nu,\sigma}(x,y) = \frac{1}{\Gamma(\sigma)} \int_0^{\infty} W_t^{\nu}(x,y) t^{\sigma-1}\, dt, \qquad x,y > 0.
$$
We consider $I^{\nu,\sigma}$ on its natural domain consisting of all $f$ for which the integral converges $x$-a.e.
Note that $I^{\nu,\sigma}$ coincides with $(B_{\nu}^{\textrm{cls}})^{-\sigma}$ defined spectrally in $L^2(d\eta_{\nu})$,
see \cite[Proposition 2.3]{NoSt1}.

In the exotic Bessel context of type $\nu < 1$ we consider the corresponding fractional integral
$\widetilde{I}^{\nu,\sigma}$ with the integral representation
$$
\widetilde{I}^{\nu,\sigma}f(x) = \int_0^{\infty} \widetilde{K}^{\nu,\sigma}(x,y) f(y) \, d\eta_{\nu}(y), \qquad x > 0,
$$
being
$$
\widetilde{K}^{\nu,\sigma}(x,y) = \frac{1}{\Gamma(\sigma)} \int_0^{\infty} \widetilde{W}_t^{\nu}(x,y) t^{\sigma-1}\, dt,
	\qquad x,y > 0.
$$
The natural domain of $\widetilde{I}^{\nu,\sigma}$ consists of all $f$ for which the defining integral converges $x$-a.e.
Formally, $\widetilde{I}^{\nu,\sigma}$ coincides in $L^2(d\eta_{\nu})$ with $(B_{\nu}^{\textrm{exo}})^{-\sigma}$ defined spectrally.
This relation can be given a strict meaning in the spirit of \cite[Proposition 2.3]{NoSt1}, but we shall not pursue this matter.

By means of \eqref{exoWcls} it is straightforward to see that
$\widetilde{K}^{\nu,\sigma}(x,y)= (xy)^{-2\nu}K^{-\nu,\sigma}(x,y)$ and, consequently,
\begin{equation} \label{Irel}
\widetilde{I}^{\nu,\sigma}f(x) = x^{-2\nu} I^{-\nu,\sigma}\big(y^{2\nu}f\big)(x), \qquad x > 0.
\end{equation}
Note that the potential kernels $K^{\nu,\sigma}(x,y)$ and $\widetilde{K}^{\nu,\sigma}(x,y)$ are infinite if
$\sigma \ge \nu +1$ and $\sigma \ge -\nu+1$, respectively, cf.\ \cite[Theorem 2.1]{NoSt1}.
 
In \cite[Theorem 2.5]{NoSt1} the following theorem was proved.
\begin{thm}[{\cite{NoSt1}}] \label{thm:potcls}
Let $\nu > -1$ and $0 < \sigma < \nu +1$. Further, let $A,B \in \mathbb{R}$ and assume that $1\le p,q \le \infty$.
\begin{itemize}
\item[(i)] $L^p(x^{Ap}d\eta_{\nu})$ is in the domain of $I^{\nu,\sigma}$ 
if and only if
\begin{equation*} %\label{cnd17}
2\sigma - \frac{2\nu+2}{p} < A < \frac{2\nu+2}{p'} \qquad \textrm{(both $\le$ when $p=1$)}.
\end{equation*}
\item[(ii)]
The estimate
\begin{equation*}
\big\|x^{-B}I^{\nu,\sigma} f\big\|_{L^q(d\eta_\nu)}\lesssim \big\|x^A f\big\|_{L^p(d\eta_\nu)}
\end{equation*} 
holds uniformly in $f \in L^p(x^{Ap}d\eta_{\nu})$ 
if and only if 
the following conditions are satisfied: 
\begin{itemize}
\item[(a)] $p \le q$,
\item[(b)] $\frac{1}{q} = \frac{1}p + \frac{A+B-2\sigma}{2\nu+2}$,
\item[(c)] $A < \frac{2\nu+2}{p'}$ \quad ($\le$ when $p = q'= 1$),
\item[(d)] $B < \frac{2\nu+2}q$ \quad ($\le$ when $p = q'=1$),
\item[(e)] $\frac{1}q \ge \frac{1}p - 2\sigma$ \quad ($>$ when $p=1$ or $q=\infty$).
\end{itemize}
\end{itemize}
\end{thm}

Combining the above result with \eqref{Irel} and Proposition \ref{thm:trans} we get the following counterpart in the exotic case.
\begin{thm} \label{thm:potstr}
Let $\nu < 1$ and $0 < \sigma < -\nu +1$. Let $A,B \in \mathbb{R}$ and assume that $1\le p,q \le \infty$.
\begin{itemize}
\item[(i)]
$L^p(x^{Ap}d\eta_{\nu})$ is in the domain of $\widetilde{I}^{\nu,\sigma}$ if and only if
$$
2\sigma - \frac{2\nu+2}{p}+2\nu < A < \frac{2\nu+2}{p'} - 2\nu \qquad (\textrm{both}  \le  \textrm{when}\; p=1).
$$
\item[(ii)]
The operator $\widetilde{I}^{\nu,\sigma}$ is bounded from $L^p(x^{Ap}d\eta_{\nu})$ to $L^q(x^{-Bq}d\eta_{\nu})$
if and only if the following conditions are satisfied:
\begin{itemize}
\item[(a)] $p \le q$,
\item[(b)] $A+B-2\sigma = (2\nu+2)\Big(\frac{1}{q}-\frac{1}{p}\Big)$,
\item[(c)] $A < \frac{2\nu+2}{p'}-2\nu \qquad (\le \textrm{when}\; p=q'=1)$,
\item[(d)] $B < \frac{2\nu+2}q-2\nu \qquad (\le \textrm{when}\; p=q'=1)$,
\item[(e)] $\frac{1}q-\frac{1}p \ge -2\sigma \qquad (> \textrm{when}\; p=1 \; \textrm{or}\;\, q=\infty)$.
\end{itemize}
\end{itemize}
\end{thm}

Note that, in view of Theorem \ref{thm:potstr}, for each $0 < \sigma < -\nu +1$ there are always non-trivial
two power weight $L^p-L^q$ inequalities for $\widetilde{I}^{\nu,\sigma}$. 
On the other hand, in the unweighted case $A=B=0$, a necessary condition for (a)--(e) in Theorem \ref{thm:potstr} to hold, with some
$0 < \sigma < -\nu + 1$, is $\nu \ge -1/2$.
Therefore, still in the unweighted case, the question about weak type endpoint estimate,
i.e.\ $L^1(d\eta_{\nu})-L^{q,\infty}(d\eta_{\nu})$ boundedness
when $1\slash q= 1-\sigma\slash (\nu+1)$ (cf.\ \cite[Theorem 2.2(iii)]{NoSt1}), makes sense only for $-1/2 \le \nu < 0$
(notice that $L^1(d\eta_{\nu})$ is not contained in the domain of $\widetilde{I}^{\nu,\sigma}$ when $\nu > 0$, see Theorem
\ref{thm:potstr}(i)).
Then the $L^1-L^{q,\infty}$ boundedness follows from the corresponding result in the classical Bessel setting and the control
$0 < \widetilde{K}^{\nu,\sigma}(x,y) < K^{\nu,\sigma}(x,y)$ (this control actually holds for all $-1< \nu < 0$,
in view of an analogous relation for the heat kernels, see \cite[Inequality (23)]{NSS}).

%%%%%%%%%%%%%%%%%%%%%%%%%%%%%%%%%%%%%%%%%%%%%%%%%%%%%%%%%%%%%%


\begin{thebibliography}{99}

\bibitem{AM}
K.F.\ Andersen, B.\ Muckenhoupt,
\emph{Weighted weak type Hardy inequalities with applications to Hilbert transforms and maximal functions},
Studia Math.\ 72 (1982), 9--26. 

\bibitem{ALN}
A.\ Arenas, E.\ Labarga, A.\ Nowak,
\emph{Exotic multiplicity functions and heat maximal operators in certain Dunkl settings},
Integral Transforms Spec.\ Funct.\ 29 (2018), 771--793. 

\bibitem{Bet14}
J.J.\ Betancor, A.J.\ Castro, J.\ Curbelo,
\emph{Spectral multipliers for multidimensional Bessel operators},
J.\ Fourier Anal.\ Appl.\ 17 (2011), 932--975. 

\bibitem{Bet12}
J.J.\  Betancor, A.J.\ Castro, J.\ Curbelo,
\emph{Harmonic analysis operators associated with multidimensional Bessel operators},
Proc.\ Roy.\ Soc.\ Edinburgh Sect.\ A 142 (2012), 945--974.

\bibitem{Bet13}
J.J.\ Betancor, A.J.\ Castro, A.\ Nowak,
\emph{Calder\'on–Zygmund operators in the Bessel setting},
Monatsh.\ Math.\ 167 (2012), 375--403.

\bibitem{Bet11}
J.J.\ Betancor, E.\  Dalmasso, J.C.\ Fari\~na, R.\ Scotto,
\emph{Bellman functions and dimension free $L^p$-estimates for the Riesz transforms in Bessel settings},
Nonlinear Anal.\ 197 (2020), 111850.

\bibitem{Bet15}
J.J.\ Betancor, J.\ Dziuba\'nski, J.L.\ Torrea,
\emph{On Hardy spaces associated with Bessel operators},
J.\ Anal.\ Math.\ 107 (2009), 195--219.

\bibitem{Bet16}
J.J.\ Betancor, J.C.\ Fari\~na, T.\ Mart{\'i}nez, L.\ Rodr{\'i}guez-Mesa,
\emph{Higher order Riesz transforms associated with Bessel operators},
Ark.\ Mat.\ 46 (2008), 219--250.

\bibitem{BHNV}
J.J.\ Betancor, E.\ Harboure, A.\ Nowak, B.\ Viviani,
\emph{Mapping properties of fundamental operators in harmonic analysis related to Bessel operators},
Studia Math.\ 197 (2010), 101--140.

\bibitem{BS}
A.N.\ Borodin, P.\ Salminen,
\emph{Handbook of Brownian motion: facts and formulae},
2nd ed., Birkh\"auser Verlag, Basel-Boston-Berlin, 2002.

\bibitem{CaSz}
A.J.\ Castro, T.Z.\ Szarek,
\emph{Calder\'on-Zygmund operators in the Bessel setting for all possible type indices},
Acta Math.\ Sin.\ (Engl.\ Ser.) 30 (2014), 637--648.

\bibitem{CRH}
A.\ Chicco Ruiz, E.\ Harboure,
\emph{Weighted norm inequalities for heat-diffusion Laguerre's semigroups},
Math.\ Z.\ 257 (2007), 329--354.

\bibitem{DPW}
J.\ Dziuba\'nski, M.\ Preisner, B.\ Wr\'obel,
\emph{Multivariate H\"ormander-type multiplier theorem for the Hankel transform},
J.\ Fourier Anal.\ Appl.\ 19 (2013), 417--437. 

\bibitem{HSTV}
E.\ Harboure, C.\ Segovia, J.L.\ Torrea, B.\ Viviani,
\emph{Power weighted $L^p$-inequalities for Laguerre-Riesz transforms},
Ark.\ Mat.\ 46 (2008), 285--313. 

\bibitem{KaPr}
E.\ Kania, M.\ Preisner,
\emph{Sharp multiplier theorem for multidimensional Bessel operators},
J.\ Fourier Anal.\ Appl.\ 25 (2019), 2419--2446.

\bibitem{Ka}
E.\ Kania-Strojec,
\emph{The atomic Hardy space for a general Bessel operator}.
\texttt{arXiv:2004.14434}.

\bibitem{Leb}
N.N.\ Lebedev, 
\emph{Special functions and their applications}. 
Dover publications, New York, 1972.

\bibitem{MST}
R.\ Mac{\'i}as, C.\ Segovia, J.L.\ Torrea,
\emph{Heat-diffusion maximal operators for Laguerre semigroups with negative parameters},
J.\ Funct.\ Anal.\ 229 (2005), 300--316. 

\bibitem{MuSt}
B.\ Muckenhoupt, E.M.\ Stein,
\emph{Classical expansions and their relation to conjugate harmonic functions},
Trans.\ Amer.\ Math.\ Soc.\ 118 (1965), 17--92.

\bibitem{lib}
\emph{NIST Digital Library of Mathematical Functions}.
\texttt{https://dlmf.nist.gov}

\bibitem{NoSj1}
A.\ Nowak, P.\ Sj\"ogren,
\emph{The multi-dimensional pencil phenomenon for Laguerre heat-diffusion maximal operators},
Math.\ Ann.\ 344 (2009), 213--248. 

\bibitem{NSS}
A.\ Nowak, P.\ Sj\"ogren, T.Z.\ Szarek,
\emph{Maximal operators of exotic and non-exotic Laguerre and other semigroups associated
with classical orthogonal expansions},
Adv.\ Math.\ 318 (2017), 307--354.

\bibitem{NoSt}
A.\ Nowak, K.\ Stempak,
\emph{Weighted estimates for the Hankel transform transplantation operator},
Tohoku Math.\ J.\ 58 (2006), 277--301.

\bibitem{NoSt1}
A.\ Nowak, K.\ Stempak,
\emph{Potential operators associated with Hankel and Hankel-Dunkl transforms},
J.\ Anal.\ Math.\ 131 (2017), 277--321.

\bibitem{handbook}
F.W.J.\ Olver, D.W.\ Lozier, R.F.\ Boisvert, C.W.\ Clark,
\emph{NIST handbook of mathematical functions},
U.S.\ Department of Commerce, National Institute of Standards and Technology, Washington, DC;
Cambridge University Press, Cambridge, 2010. 

\bibitem{RY}
D.\ Revuz, M.\ Yor,
\emph{Continuous martingales and Brownian motion}.
Grundlehren der Mathematischen Wissenschaften [Fundamental Principles of Mathematical Sciences], 293.
Springer-Verlag, Berlin, 1991.

\bibitem{Wat}
G.\ N.\ Watson,
\emph{A treatise on the theory of Bessel functions},
Cambridge University Press, Cambridge, 1966.

\end{thebibliography}
\end{document}